\documentclass[11pt, article]{amsart}
\usepackage{euscript,amscd,amsgen,amsfonts,amssymb, amsmath, latexsym,graphicx,psfrag,pgfplots, multicol}
\usepackage{stmaryrd, mathrsfs,tikz-cd}
\usepackage{bbm}
\usepackage{url}

\newcommand{\eqdef}{\stackrel{\scriptscriptstyle\rm def}{=}}

\newtheorem{theorem}{Theorem}

\newtheorem{corollary}{Corollary}
\newtheorem*{corollary*}{Corollary}
\newtheorem{definition}{Definition}
\newtheorem{lemma}{Lemma}
\newtheorem*{remark}{Remark}

\newtheorem{example}{Example}

\newtheorem*{definition*}{Definition}
\newtheorem*{thma}{Theorem A}
\newtheorem*{thmb}{Theorem B}
\newtheorem*{thmc}{Theorem C}

\newcommand{\llb}{\llbracket}
\newcommand{\rrb}{\rrbracket}

\newcommand{\con}{{\rm con}}
\newcommand{\res}{{\rm res}}

\renewcommand{\top}{{\rm top}}

\newcommand{\cM}{\EuScript{M}}

\newcommand{\bR}{{\mathbb R}}

\newcommand{\bZ}{{\mathbb Z}}
\newcommand{\bN}{{\mathbb N}}

\newcommand{\cF}{{\mathcal F}}
\newcommand{\cA}{{\mathcal A}}

\newcommand{\cG}{{\mathcal G}}
\newcommand{\cL}{{\mathcal L}}

\newcommand{\bi}{{\mathbbm i}}
\newcommand{\bl}{{\mathbbm l}}

\title{Ergodic theory on coded shift spaces}
\author[T. Kucherenko]{Tamara Kucherenko}\address{Department of Mathematics, The City College of New York, New York, NY, 10031, USA}\email{tkucherenko@ccny.cuny.edu}

\author[M. Schmoll]{Martin Schmoll}
\address{Department of Mathematical Sciences, Clemson University, Clemson SC}\email{schmoll@clemson.edu}

\author[C. Wolf]{Christian Wolf}\address{Department of Mathematics, The City College of New York and The Graduate Center of CUNY, NY, 10031, USA}\email{cwolf@ccny.cuny.edu}

\begin{document}

\begin{abstract}
We study ergodic-theoretic properties of coded shift spaces. A coded shift space is defined as a closure of all bi-infinite concatenations of words from a fixed countable generating set. We derive sufficient conditions for the uniqueness of measures of maximal entropy and equilibrium states of H\"{o}lder continuous potentials based on the partition of the coded shift into its concatenation set (sequences that are concatenations of generating words) and its residual set (sequences added under the closure). In this case we provide a simple explicit description of the measure of maximal entropy.
We also obtain
flexibility results for the entropy on the concatenation and residual sets. Finally, we prove a local structure theorem for intrinsically ergodic coded shift spaces which shows that our results apply to a larger class of coded shift spaces compared to previous works by Climenhaga \cite{Cl}, Climenhaga and Thompson \cite{ClTh, ClTh2}, and Pavlov \cite{Pavlov}.
\end{abstract}
\maketitle

\section{Introduction}

\subsection{Motivation}
It is one of the central goals in the ergodic theory of dynamical systems to identify the typical dynamics of a given system. In this context the word `typical' refers to a set of points of full measure with respect to an invariant probability measure. For many classes of dynamical systems the set of invariant probability measures is rather large which raises the question which invariant measures are the natural choice to use as reference measures. Arguably the most natural reference measure choice is a measure
whose typical points carry the entire topological orbit complexity in terms of entropy. Such a measure (if it exists) is called a measure of maximal entropy. If a measure of maximal entropy exists and is unique the dynamical system is called intrinsically ergodic \cite{Wei1, Wei2}.

In this paper, we work in the context of symbolic dynamical systems over a finite alphabet. There is a vast amount of research on the questions concerning the uniqueness of measures of maximal entropy and equilibrium states, e.g., \cite{jB05,Cl,ClTh,ClTh2,GRP,H,Krieger,Pavlov}.
We study coded shift spaces which were introduced by Blanchard and Hansel \cite{BH} as  a natural generalization of transitive sofic shifts. Coded shifts include a wide range of well-known classes of symbolic systems, e.g., S-gap shifts, generalized gap shifts, and beta shifts. Our goal is to develop a fairly complete theory about the uniqueness of measures of maximal entropy and equilibrium states for H\"older continuous potentials. Roughly speaking, our classification depends on whether the entropy (or pressure) is concentrated on the concatenation set or the residual set. In the entropy case our results do not only guaranty the uniqueness of the measure of maximal entropy, but also provide an explicit formula for the values of such measure on cylinders. In particular, we obtain new information about the structure of measures of maximal entropy for such well-studied shift spaces as beta shifts and the Dyck shift.   

While evaluating the entropy and pressure on the concatenation and the residual sets we also derive flexibility results concerning their entropies. Moreover, we establish a local structure theorem for intrinsically
ergodic coded shifts. From the structure theorem we conclude that our results apply to an abundance of coded shift spaces for which previous results of Climenhaga \cite{Cl}, Climenhaga and Thompson \cite{ClTh, ClTh2}, and Pavlov \cite{Pavlov} do not provide answers about the uniqueness of measures of maximal entropy and equilibrium states.

\subsection{Main results}
Let $X$ be a bi-infinite shift space over a finite alphabet $\cA$. This means that $X$ is a closed (with respect to the Tychonoff product topology), shift invariant subset of the full shift $\cA^\bZ$. Denote the left shift map on $X$ by $\sigma$. We say $(X,\sigma)$ is a coded shift space (or simply $X$ is coded), if there exists a generating set $\cG$ which is a subset of the set of all finite words over $\cA$
such that $X = X(\cG)$ is the smallest shift space that contains the concatenation set $X_{\rm con}=X_{\rm con}(\cG)$, i.e. all bi-infinite concatenations of generators. In other words, $X$ is the topological closure of $X_{\rm con}$. We refer to \cite{LindMarcus} for equivalent ways to define coded shifts.
We call the set $X_{\rm res}=X_{\rm res}(\cG)=X\setminus X_{\rm con}$ the residual set of $X$. We refer to Section 3 for more details. We recall, that if the generating set is finite, then $X$ is a transitive sofic shift. In particular, its theory of measures of maximal entropy and equilibrium states for H\"older continuous potentials is well-understood. Therefore, the focus of this paper is the case of infinite generating sets.

A priori, it is possible that points in $X_{\rm con}$ have multiple representations as concatenations of generators. We say a coded shift $X$ is uniquely representable, if  there exists a generating set $\cG$ such that $X=X(\cG)$ and every point in $X_{\rm con}(\cG)$ has a unique representation as a concatenation of generators. In this situation we also say that $\cG$ uniquely represents $X_{\rm con}(\cG)$. 
A weaker condition is that $X=X(\cG)$ has the unique decipherability property, i.e., every finite word in $X$ has at most one representation as a concatenation of generators. It is  shown by Blanchard and Hansel \cite{BH} in 1986 that every coded shift has a generating set $\cG$ which has the unique decipherability property.
By using results from \cite{FF}, B\'{e}al, Perrin and Restivo \cite{BPR} recently proved that every coded shift is uniquely representable. 
Moreover, the proof in \cite{BPR} is constructive. In particular, the proof can be used as a blueprint for a computer program which, based on the input of an arbitrary generating set $\cG'$ outputs a generating set $\cG$ such that $\cG$ uniquely represents $X_{\rm con}(\cG)$ and $X(\cG)=X(\cG')$. Based on this result, it is reasonable to assume that we have a generating set which uniquely represents its concatenation set.

We remark that the set of coded shifts spans a rather broad collection of shift spaces. Indeed, every shift space $X$ can be
expressed by a representation of $X$, i.e., a countable directed labeled graph such that $X$ is the closure of the set of bi-infinite sequences which are associated with bi-infinite paths on the graph. Coded shifts are precisely those shift spaces which have a strongly connected representation, see \cite{LindMarcus}. It turns out that many well-known classes of shifts are coded including S-gap shifts, generalized gap shifts, transitive Sofic shifts, beta shifts, and many more, see, e.g., \cite{BDWY}.

Let $X$ be a shift space and let $h_{\rm top}(X)$ denote the topological entropy of $\sigma\vert_X$. The variational principle for the topological entropy states that $h_{\rm top}(X)=\sup_{\mu\in \cM_\sigma(X)} h_\sigma(\mu)$ where $\cM_\sigma(X)$ is the space of all Borel $\sigma$-invariant probability measures on $X$ endowed with the weak$^\ast$ topology, and $h_\sigma(\mu)$ denotes the measure-theoretic entropy of $\mu$ defined in \eqref{eqn:def:meas_ent}. 
We say a measure $\mu\in \cM_\sigma(X)$ is a measure of maximal entropy for $(X,\sigma)$ if $h_{\rm top}(X)=h_\sigma(\mu).$ It follows from the expansivity of the shift map that there  exists at least one (ergodic) measure of maximal entropy. However, the measure of maximal entropy is, in general, not unique \cite{Sh,Krieger,jB05,H}; if it is unique the system is called intrinsically ergodic.
Roughly speaking, if an ergodic measure $\mu$ is a measure of maximal entropy its typical orbits carry the full orbit complexity in terms of entropy.

Our approach is to study the entropy  on the sets  $X_{\rm con}$ and $X_{\res}$ separately. We define
 the \emph{concatenation entropy} of  $X$ by
 \begin{equation}\label{secentropy}
h_{\rm con}(X)= \sup\{h_\sigma(\mu): \mu\in\cM_\sigma(X)\text{ and } \mu(X_{\rm con})=1\},
 \end{equation}
 and the \emph{residual entropy} of $X$ is defined by
 \begin{equation}\label{resentropy}
  h_{\rm res}(X)=\sup\{h_\sigma(\mu): \mu\in\cM_\sigma(X)\text{ and } \mu(X_{\rm res})=1\}.
 \end{equation} 
 We note that the concatenation entropy and residual entropy were introduced in \cite{BDWY} where they were called sequential entropy and limit entropy, respectively.
It is shown in Lemma \ref{gcylindermeas} that $X_{\rm con}$ and $X_{\res}$ are Borel sets and thus $h_{\rm con}(X)$ and $h_{\rm res}(X)$ are well defined. Clearly, $X_{\rm con}$ and $X_{\rm res}$ are shift invariant sets.
Therefore, $h_{\rm top}(X)=\max\{h_{\rm con}(X), h_{\rm res}(X)\}$. Moreover, every ergodic measure of maximal entropy must put full measure either on $X_{\rm con}$ or on $X_{\rm res}$. We refer to the analogous definitions of the concatenation respectively residual pressure to Section \ref{sec:unieqsta}.

Our first main result concerns the flexibility of the concatenation entropy. It is not too hard to see that every invariant measure which puts full measure on $X_{\rm res}$ must also put full measure on a certain subshift called the limit set $X_{\rm lim}$ which is the set of points $x\in X$ 
such that every subword of $x$ is a subword of some generator in $\cG$ (see \eqref{def:X_lim} for the precise definition). Since the language of the concatenation set contains the language of the limit set one might heuristically expect that $h_{\rm con}(X)\geq h_{\rm res}(X)$ does hold in general. The problem with this heuristic argument is that $X_{\rm con}$ is not closed.
In particular, the next theorem  (see Theorem \ref{thm:construstion} in the text) shows that if a fixed subshift $Z$ is contained in the residual set, the concatenation entropy might still be arbitrary small.

\begin{thma}  Suppose $Z$ is a transitive subshift on a finite alphabet $\cA$ which is not the full shift. Then there exists a generating set $\cG$ on the same alphabet $\cA$ such that for $X=X(\cG)$ the following holds:
  \begin{enumerate}
    \item[(i)] $\cG$ uniquely represents $X_{\rm con}$;
    \item[(ii)] $Z\subset X_{\rm res}$, and if $\mu\in \cM_\sigma(X)$ with $\mu(X_{\rm res})=1$ then $\mu(Z)=1$;
    \item[(iii)] $h_\con(X)<\epsilon$.
  \end{enumerate}
\end{thma}
Theorem A shows that  it is not possible to obtain a general uniqueness result for measures of maximal entropy for the case $h_{\rm res}(X)>h_{\rm con}(X)$. Indeed, if the subshift $Z$ in Theorem A is of positive topological entropy and has multiple measures of maximal entropy (e.g., one of the shifts constructed in \cite{jB05,H,Sh}) then for $0<\epsilon<h_{\rm top}(Z)$ the coded shift $X$ also has multiple measures of maximal entropy (which are precisely the measures which maximize the entropy on $Z$). 
We note here that the main difficulty is to prove property (iii). Since $X_\con$ is not a subshift, estimates on word complexity cannot be applied. To overcome this difficulty we design a certain coded shift $Y$ with $h_{\rm con} (Y)>  h_{\rm res}(Y)$ such that every invariant measure on $X$ which puts full measure on the concatenation set is measure-theoretic isomorphic to a measure on the concatenation set of $Y$. This allows us to deduce (iii) from the corresponding property of $Y$.

Next we consider the case $h_{\rm res}(X)=h_{\rm con}(X)$. In this case it is possible to produce examples of coded shifts $X$ with multiple measures of maximal entropy, see Examples \ref{ex:hseq=hres} and \ref{ex:Dyck}.


Finally, we consider the case $h_{\rm con}(X)>h_{\rm res}(X)$ and establish the uniqueness of the measure of maximal entropy.  To characterize this measure we introduce and develop in Section \ref{sec:G-topology} the theory of the so-called $\cG$-Bernoulli measures on a coded shifts $X$ with generating set $\cG$.  These are probability measures which assign full measure to the concatenation set and share some similarities with Bernoulli measures which contain the unique measures of maximal entropy of full shifts. Precisely, we say that an invariant probability measure $\mu$ is $\cG$-Bernoulli if there exist positive real numbers $p_g, g\in \cG$ with $\sum_{g\in \cG} p_g=1$ and $c>0$ such that
\[
\llb g_{0}\dots g_{k}\rrb=\frac{1}{c}\, p_{g_{0}}\cdot\ldots\cdot p_{g_{k}}
\]
for all $g_{0},\dots, g_{k}\in \cG$.  Here $\llb g_{0}\dots g_{k}\rrb$ denotes the $\cG$-cylinder consisting of all points  $x=\cdots g'_{{-1}}.g'_{0}\cdots g'_{k}g'_{{k+1}}\cdots\in X_{\rm con}$ such that $g'_{j}=g_{j}$ for $j=0,\dots,k$. We refer to Section \ref{sec:G-topology} for further details about $\cG$-cylinders and $\cG$-Bernoulli measures. With these definitions in mind we have the following result.

\begin{thmb}Let $X=X(\cG)$ be a coded shift such that $\cG$ uniquely represents $X_{\con}$ and $h_{\con}(X)>h_{\res}(X)$.
Then $X$ has a unique measure of maximal entropy $\mu_{\max}$. Moreover, $\mu_{\max}$  is
$\cG$-Bernoulli with $p_g=\exp(-|g|h_\top(X))$ and $c=\sum_{g\in \cG} |g| \exp(-|g|h_\top(X))$.
\end{thmb}
It is important to note that Theorem B does not only establish the uniqueness of the measure of maximal entropy for a large class of shift spaces, it also provides an explicit formula for this measure.  In the context of symbolic dynamics such an explicit formula for the measure of maximal entropy has, to the best of our knowledge, previously only been known for subshifts of finite type, sofic shifts, S-gap shifts \cite[Theorem 3.22]{GRP}, and beta shifts \cite{Hof}. In the latter case the measure of maximal entropy is the lift of the Parry measure on a countable full shift.
As a consequence of Theorem B and results in \cite{BDWY} we obtain that the unique measure of maximal entropy of a beta shift is the  $\cG$-Bernoulli measure with respect to the natural generating set $\cG$.
 We refer to Example \ref{Example4} for more details.

 Furthemore, the techniques used in the proof of Theorem B can in some cases be applied to obtain an explicit formula for a measure of maximal entropy even when it is not unique. The idea is to use different adaptations of the generating set for each measure of maximal entropy of the coded shift $X$ and show that the said measure is $\cG$-Bernoulli for the new generating set. We showcase this approach on the Dyck shift in Example \ref{ex:Dyck}. The Dyck shift is the go-to example of a coded shift where the assumptions of the uniqueness theorems in \cite{Cl,Pavlov}, as well as Theorem B, fail: for its canonical generating set we have $X\subset X_{\rm lim}$ and $h_\con(X)<h_\res(X)$. Also, the Dyck shift does have two ergodic measures  of maximal entropy. Despite these pathological properties, we are able to show that both these measures are $\cG$-Bernoulli with respect to slightly modified generating sets.

We recall that for transitive subshifts of finite type and H\"older continuous potentials the unique equilibrium state is a Gibbs measure (see \eqref{defGibbs} for the definition) which is Bernoulli \cite{Bo}. It turns out that in the case of intrinsic ergodicity of a coded shift $X$, the measure of maximal entropy $\mu_{\rm max}$ in Theorem B does not, in general, have the Gibbs property (see Example \ref{ex:notGibbs}).

 As a consequence of Theorem B we obtain a formula for the topological entropy of a large class of coded shifts.
Namely, under the assumptions of Theorem
B we obtain that $h_\top(X)=\log \lambda_\ast$ where $\lambda_\ast$ is the unique solution of
the characteristic equation \begin{equation}\label{entcoded}
\sum_{g\in \cG} \lambda^{-|g|}=1,
\end{equation}
see Lemma \ref{lem:Gurevich}.
We note that \eqref{entcoded}  generalizes several papers in the literature with similar formulas for the topological entropy of certain classes of coded shift spaces, e.g., \cite{ADJS, BDWY, Cl2, Dillon, MS, Pavlov, Spandl}. The origin of \eqref{entcoded} can be traced back to the loop method for computing the entropy of a subshift of finite type developed by Peterson \cite{Petersen} in 1986, and, in particular, \eqref{entcoded} appears in \cite[Theorem 1.3]{Pavlov} under a special condition on the coded shift space.

One consequence of the entropy formula \eqref{entcoded}  is that $X$ is an almost-sofic shift, that is, there exists an increasing sequence of sofic shifts $X_m\subset X$ such that $\lim_{m\to\infty}h_{\rm top} (X_m)=h_{\rm top}(X)$. In fact, in the setting of Theorem B one can take $X_m=X_m(\cG(m))$ where $\cG(m)=\{g_1,\dots,g_m\}\subset \cG.$ We refer to \cite{Petersen} for more details about almost-sofic shifts.

It is a well-known fact that transitive sofic shifts satisfy the entropy-minimality property, i.e., every proper subshift must have strictly smaller topological entropy, see, e.g., \cite{LindMarcus}. We show  that this result does not generalize to almost sofic shifts. Namely, we construct in Example \ref{ex:hseq=hres} an almost sofic coded shift $X$ which contains two disjoint proper subshifts $Z_1$ and $Z_2$ with $h_{\rm top}(Z_1)=h_{\rm top}(Z_2)=h_{\rm top}(X)$.

Recall that $X_{\rm con}$ is not closed and in particular not a shift space. Therefore, Theorem B combined with the fact that $X_{\rm con}$ is not a subshift
implies the following:
\begin{corollary*}
     Suppose a coded shift $X$ has a measure of maximal entropy which assigns full measure to a subshift $Y\subset X$ with $Y\not=X$. Then $h_{\rm con}(X)\leq h_{\rm res}(X)$ for any generating set $\cG$ with $X=X(\cG)$, and which uniquely represents $X_{\rm con}(\cG)$.
\end{corollary*}
\noindent We note that the corollary above 
is sharp. This follows from Example \ref{ex:hseq=hres}.

We also consider the uniqueness of equilibrium states of hyperbolic H\"older continuous potentials. In particular, we show that if the entropy $h(\cG)$ of the generating set $\cG$ satisfies
\[h(\cG)\eqdef \limsup_{n\to \infty}  \frac{1}{n}\log |\{g\in \cG: |g|=n\}|=0\]
and the pressure of a hyperbolic H\"older continuous potential $\phi$ on the concatenation set is greater than the pressure of $\phi$ on the residual set, then $\phi$ has a unique equilibrium state which is the lift of the invariant Gibbs measure on the associated induced system. We refer to Theorem
\ref{thm:uniqueeqsta} for more details.

As mentioned above, Theorem B and Theorem \ref{thm:uniqueeqsta} are not the first results about unique measures of maximal entropy and equilibrium states for coded shifts. Indeed, Climenhaga \cite{Cl}, Climenhaga and Thompson \cite{ClTh,ClTh2} and Pavlov \cite{Pavlov} obtained uniqueness results under the assumption that the topological entropy  (respectively pressure) of $X$ exceeds the entropy of the limit set $X_{\rm lim}$. From the measure theoretical perspective this means that the concatenation entropy (respectively pressure) is greater than that of the limit set.
The approach of the above works is to generalize the classical specification-based methods of Bowen by considering weaker versions of specification. 
 They successfully establish and study the unique measure of maximal entropy (respectively equilibrium state) in this case.
 
  A key difficulty for the applicability of the Climenhaga, Thompson, Pavlov \cite{Cl, ClTh,ClTh2,Pavlov} results is that  $X_{\rm con}$ and $X_{\rm lim}$ are in general not disjoint. In fact, for many classes of coded shift spaces the limit set contains the concatenation set and consequently the results in \cite{Pavlov,ClTh, ClTh2,Cl} do not apply to obtain results about the uniqueness of entropy maximizing measures or equilibrium states.
 We illustrate this in Theorem C below, which is a combination of Theorem \ref{thm:local_structure} and Corollary \ref{lem:ClPav} in Sections \ref{sec:flexibility} and \ref{sec:mme}, respectively. We show that for any coded shift space $X$ with $h_{\rm top}(X)>h_{\rm top}(X_{\rm lim})$, there are infinitely many coded shift spaces $\widetilde{X}$ "close" to $X$ such that   $\widetilde{X} \subset \widetilde X_{\lim} $ holds while Theorem B still applies to $\widetilde{X}$ and guarantees the uniqueness of the measure of maximal entropy. Consequently, Theorem  B significantly extends the results of Climenhaga, Thompson and Pavlov.
  \begin{thmc}
Let $X=X(\cG)$ be a coded shift such that
$\cG$ uniquely represents $X_{\con}$ and $h_{\rm top}(X)>h_{\rm top}(X_{\rm lim})$.
Then for every $\epsilon>0$ there exists a generating set $\widetilde{\cG}$ such that the coded shift $\widetilde X = X(\widetilde \cG)$ contains $X$, and has the following properties:
\begin{enumerate}
\item[(i)]   $\widetilde \cG$ uniquely represents $\widetilde X_{\con}$;
\item[(ii)] $\widetilde{X} \subset \widetilde X_{\lim} $;
\item[(iii)] $h_{\res}(X)=h_{\res}(\widetilde X)<h_\con(\widetilde X)$;
\item[(iv)] $h_{\rm top}(X)<h_{\top}(\widetilde X)<h_{\top}(X)+\epsilon$.
\end{enumerate}
\end{thmc}
 
We note that Theorem \ref{thm:local_structure} is in fact a more general result than Theorem C. In some sense this theorem can be considered as a local structure result for the concatenation and residual entropy of coded shift spaces.

The paper is organized as follows.  In Section 2 we review some
basic concepts from symbolic dynamics and  the thermodynamic formalism.
 Section 3 is devoted to the study of the concatenation and residual entropy for coded shift spaces. In particular, we present the proof of the flexibility result Theorem A and the local structure result Theorem \ref{thm:local_structure}. In Section 4, we introduce and study $\cG$-cylinders and so-called $\cG$-Bernoulli measures for coded shift spaces. We also show that both the concatenation set and the residual set are Borel. The proof of  Theorem B concerning the uniqueness of the measure of maximal entropy for coded shift spaces is presented in Section 5. In addition, several examples are constructed which emphasize various aspects of our results. Lastly, in Section 6 we derive results for the uniqueness of equilibrium states for hyperbolic H\"older continuous potentials.

\section {Preliminaries}\label{sec:prelim}
We continue to use the notation from Section 1. We consider shift spaces over a (fixed) finite alphabet $\cA$.
Let $\cA^*=\bigcup_{n\in\bN_0}\cA^n$ be the set of all finite words over $\cA$. The length of the word $w\in\cA^*$ is denoted by $|w|$, whereas the empty word has length zero. Given $u,v\in \cA^*$ we write $uv$ for the concatenation of the words $u$ and $v$. For $x=(x_i)_{i\in\bZ}\in\cA^{\bZ}$ and integers $n<m$ we use the notation $x[n,m]$ for the word $x_n\dots x_m$ and $x[n,m)$ for the word $x_n...x_{m-1}$. With a slight abuse of notation we also allow $n=-\infty$ and $m=\infty$, i.e., we consider the rays
$x(-\infty,m], x(-\infty,m), x[n,\infty]$ and $x(n,\infty)$.
The cylinder of the word $w\in\cA^*$ is the subset of $\cA^{\bZ}$ given by $[w]=\{x\in\cA^\bZ: x[0,|w|)=w\}$.

We briefly review some relevant definitions from the thermodynamic formalism, see \cite{MU,Wal:81} for detailed accounts.
Let $X$ be a shift space with shift map $\sigma:X\to X$. 
Endowing $X$ with the {\em Tychonov product topology} makes $X$ into a compact  and metrizable topological space. In fact, the (standard) metric given by
\begin{equation}\label{defmetX}
d(x,y)\eqdef  2^{-\min\{|k| \;:\;  x_k\neq y_k\}} 
\end{equation}
induces the Tychonov product topology on $X$.
We denote  the set of $X$-admissible words of length $n$ by $\cL_n(X)$.  We  call $\cL(X)\eqdef \bigcup_{n=0}^\infty \cL_n(X)$ the {\em language of $X$}. 
Given $\mu\in \cM_\sigma(X)$,  the \emph{measure-theoretic entropy} of $\mu$  is defined by
 \begin{equation}\label{eqn:def:meas_ent}
 h_\sigma(\mu)=\lim_{n\to\infty}- \frac{1}{n}\sum_{w\in \cL(X,n)} \mu([w])\log (\mu([w])),
 \end{equation}
 omitting terms with  $\mu([w])=0$.  
 Let $\phi\in C(X,\bR)$.  For $n\geq 1$, we define the {\em $n$-th partition function} $Z_n(\phi)$ at $\phi$ by
\[
Z_n(\phi)=\sum_{w\in \cL_n(X)}\exp\left( \sup_{x\in [w]} S_n\phi (x) \right), 
\]
where  
\[
S_n \phi(x) \eqdef \sum_{k=0}^{n-1} \phi(\sigma^k(x)).
\]
We observe that the sequence $\left(\log Z_n(\phi)\right)_{n\geq 1}$ is subadditive, see, e.g., \cite{MU}. The \emph{topological pressure} of $\phi$ with respect to the shift space $X$ is defined by 
\begin{equation} \label{eqn:def:P}
P_{\rm top}(X,\phi)= \lim_{n\to \infty} \frac{1}{n} \log Z_n(\phi)=\inf \left\{\frac{1}{n} \log Z_n(\phi): n\geq 1\right\}.
\end{equation}
Moreover, $h_{\rm top}(X)=P_{\rm top} (X,0)$ denotes the {\em topological entropy} of $\sigma$.
We sometimes consider the topological pressure simultaneously on different shift spaces $X$, i.e., $X$ may be a (one-sided or two-sided) shift space over a finite or countable infinite alphabet.  However, when it is clear from the context we suppress $X$ in the notation of the pressure and write $P_{\rm top}(\phi)$ for $P_{\rm top}(X,\phi)$.

 Let $\phi\in C(X,\bR)$. Recall that a Borel probability measure $m$ on $X$ is a Gibbs measure for $\phi$  if there exist constants $M>1$ and $P_m$ such that 
for all $n\in \bN$, $w\in\cL_n(X)$, and  $x\in [w]$ 
\begin{equation}\label{defGibbs} M^{-1} \leq \frac{m([w])}{\exp(S_n\phi(x)-n P_m)}\leq M.\end{equation}
Moreover, for every Gibbs measure $m$ of $\phi$, $P_m=P_{\rm top}(\phi)$. In case the potential $\phi$ is clear from the context we simply say that $m$ is Gibbs. If additionally a Gibbs measure  is  shift-invariant, we say it is an invariant Gibbs measure.
We refer to \cite{Bo} and \cite{MU} for more details.

In this paper we also make use of the measure-theoretic entropy, the topological pressure and Gibbs measures for the one-sided full shift $(Y^+,\sigma)$ and two-sided full shift $(Y,\sigma)$ over a countable infinite alphabet, i.e., $Y^+=\bN^{\bN_0}$ and $Y=\bN^\bZ$, respectively. We note that in the countable infinite alphabet case the definitions for the measure-theoretic entropy, topological pressure and Gibbs measures are identical to the definitions \eqref{eqn:def:meas_ent}, \eqref{eqn:def:P} and \eqref{defGibbs}, see \cite{MU} for details.
We note that in the countable infinite alphabet case the quantities in \eqref{eqn:def:meas_ent} and \eqref{eqn:def:P} can be infinite. In particular, the infimum in \eqref{eqn:def:P} may be taken over a set of infinities which is infinity.

\section {Flexibility of the Residual Set}\label{sec:flexibility}
We fix an arbitrary nonempty set of finite words $\cG\subset\cA^*$. The \emph{coded shift with the generating set} $\cG$ is the smallest subshift of $\cA^{\bZ}$ which contains all bi-infinite concatenations of the elements of $\cG$. We denote this subshift by $X(\cG)$. Then $X(\cG)$ can be naturally decomposed into two disjoint subsets: the \emph{concatenation set} $X_{\rm con}(\cG)$ consisting of the concatenations of the words from $\cG$ and the \emph{residual set} $X_{\rm res}(\cG)$ consisting of the points added to $X_{\rm con}(\cG)$ under the closure operation. More precisely, we define
\begin{align*}
  X_{\rm con}(\cG) & =\{x\in\cA^{\bZ}: \exists (k_i)_{i\in\bZ}\subset \bZ, k_i<k_{i+1}: x[k_i,k_{i+1})\in\cG   \} \\
  X_{\rm res}(\cG) &= X(\cG)\setminus X_{\rm con}(\cG).
\end{align*}
We say that a generating set $\cG$ uniquely represents $X_{\con}$ if for each $x\in X_{\rm con}(\cG)$ there is only one set of integers $\{k_i: i\in\bZ\}$ satisfying the above property. 
 It was shown in \cite{BPR} that we can always find a 
 generating set $\cG'$ so that $X(\cG')=X(\cG)$ and $\cG'$ uniquely represents $X_{\con}$. Therefore, we always assume that our generating set $\cG$ uniquely represents the concatenation set of the coded shift $X(\cG)$. 
 One might think about the set $\cG$ as a ``countable alphabet" for $X_{\con}(\cG)$. We use some of the analogous notations, e.g., for $W\subset\cG$ we write $W^*$ for the set of all finite concatenations of the elements in $W$. Since the generating set $\cG$ is fixed, we frequently omit it in our notation and write $X$ for $X(\cG)$, $X_{\rm con}$ for $X_{\rm con}(\cG)$, $X_{\rm res}$ for $X_{\rm res}(\cG)$, etc.

Another important subset of $X(\cG)$, which was considered by Climenhaga and Pavlov \cite{Cl, Pavlov}, is the \emph{limit set} defined by
\begin{equation}\label{def:X_lim}
  X_{\rm lim}(\cG) =\{x\in\cA^{\bZ}: \text{for any }k,\,x[-k,k]\text{ is a subword of some word in } \cG \}.
\end{equation}
Pavlov proved in \cite{Pavlov} that $\mu(X_{\rm con}\cup X_{\rm lim})=1$ for all $\mu\in \cM_\sigma(X)$, and hence the concatenation set and the limit set determine the measure-theoretic dynamics on $X$. In fact, there are a number of results (e.g., \cite{Cl,ClTh,Pavlov}) establishing various properties (including uniqueness) of measures of maximal entropy as well as equilibrium states for H\"{o}lder potentials which rely on the strict inequality between the entropy respectively pressure of $X$ and $X_{\rm lim}$. However, the set $X_{\rm lim}$ might be quite large in general, and even coincide with the set $X$ itself. This is, for example, the case for the Dyck shift (see Example \ref{ex:Dyck}), which was noted in \cite{Pavlov}. The main motivation of the present work is to derive sufficient conditions for the uniqueness of measures of maximal entropy and equilibrium states based on the disjoint partition of the coded shift $X$ into its concatenation and residual sets.

It was shown in \cite{BDWY} that in the case when the concatenation entropy is larger than the residual one  we can compute the entropy of $X$ as a limit of entropies of certain sofic subshifts. Precisely, consider any coded shift  $X(\cG)$ with an infinite generating set $\cG=\{g_1,...,g_m,\cdots\}$ which uniquely represents $X_{\con}$.
Let $X_m$ be coded shifts with generators $\{g_1,...,g_m\}$.
Note that each $X_m$ is a sofic shift since it has a finite generating set. If $h_{\rm con}(X)>h_{\rm res}(X)$ then $\lim\limits_{m\to\infty} h_{\rm top}(X_m)=h_{\rm top}(X)$. In contrast to this result we present a construction of a coded shift $X$ which has any given transitive subshift as its residual set. Moreover, the entropies of $X_m$ can be made uniformly bounded above by any given positive number, no matter how small. In fact, we prove that this number also bounds the concatenation entropy of $X$, hence providing a family of examples where a coded shift $X$ satisfies $h_{\rm con}(X)\ll h_{\rm res}(X)$.

\begin{theorem}\label{thm:construstion}
  Suppose $Z$ is an irreducible subshift on a finite alphabet $\cA$ which is not a full shift. Let $\epsilon>0$ be given. Then there exists a generating set $\cG$ on the same alphabet $\cA$ such that for $X=X(\cG)$ the following holds:
  \begin{enumerate}
    \item[(i)] $\cG$  uniquely represents $X_{\rm con}$;
    \item[(ii)] $Z\subset X_{\rm res}$, and if $\mu\in\cM_\sigma(X)$ with $\mu(X_\res)=1$ then $\mu(Z)=1$;
    \item[(iii)] $h_\con(X)<\epsilon$ 
  \end{enumerate}
\end{theorem}

To prove the theorem, we build the generating set $\cG$ in such a way that distinct generating words can not have the same length. We show that the concatenation entropy in this case is controlled by the sequence of the cardinalities of elements in $\cG$ and then vary these cardinalities to achieve the $\epsilon$-bound for the entropy.
It turns out that such statement holds for any coded shift when the number of the generating words of fixed length is uniformly bounded. We establish this fact in the next lemma by constructing an isomorphism between $X_{\rm con}$ and a certain generalized version of the $S$-gap shift on a large alphabet. Then we use properties of this shift to make assertions about the invariant measures on $X_{\rm con}$. As a nice by-product of this approach we obtain a formula for the concatenation entropy of a coded shift with bounded rate of growth of generating words.

The classical $S$-gap shift associated with a set $S$ of non-negative integers on the alphabet $\{0,1\}$ consists of strings where the number of 1's between any two nearest 0's must be an element of $S$. We note that the roles of 0 and 1 are reversed here compared to the convention in the literature, but it is more convenient for introducing the generalization we need. The entropy of the $S$-gap shift is $\log \lambda$ where $\lambda$ satisfies
\begin{equation}\label{eq:s-gap_entropy}
  \sum_{s\in S}\lambda^{-(s+1)}=1.
\end{equation}

Although the expression \eqref{eq:s-gap_entropy} for the entropy of the $S$-gap shift was widely used, for a period of several decades there was no reference to a proof. It appears as an exercise in the monograph of Lind and Marcus \cite{LindMarcus}, which would certainly discourage one from publishing such a result. There is an incomplete proof by Spandl \cite{Spandl} (fixed years later in \cite[Corollary 3.18]{GRP}) and two alternative proofs are presented in Climenhaga's blog \cite{Cl2}. 

For us, the $S$-gap shift is a coded shift with generating set $\{01^{s}: s\in S\}$. The generalization we need for our purpose is the shift $Y$ on the alphabet $\{0,...,d\}$ associated to sets $S_1,...,S_d$ of non-negative integers given by the generator $\{0i^{s_i}: s_i\in S_i,\, i=1,...,d\}$. We find its entropy by applying Pavlov's result for coded shifts \cite[Theorem 1.3]{Pavlov}, which, in particular, gives a nice illustration of the usefulness of the theory. 
It is easy to see that $Y_{\rm lim}$ consists of $d$ periodic points of the form $(...iii...)$, so its entropy is zero. Hence, $h_{\rm top}(Y)$ satisfies $\sum_{n=1}^{\infty}c(n)e^{-nh_{\rm top}(Y)}=1$ where $c(n)$ is the number of words in the generator of $Y$ of length $n$. We can rewrite this equation in the spirit of (\ref{eq:s-gap_entropy}), so that $h_{\rm top}(Y)=\log \lambda$ where $\lambda$ is the unique solution of 
\begin{equation}\label{eq:multy-gap_entropy}
  \sum_{s\in S_1}\lambda^{-(s+1)}+\cdots +\sum_{s\in S_d}\lambda^{-(s+1)}=1.
\end{equation}

\begin{lemma}\label{lem:entropy_formula}
Let $X$ be a coded shift with generator $\cG$ which uniquely represents $X_{\con}$. Suppose $\cG$ has a bounded rate of growth, i.e. there is an integer $d$ such that ${\rm card}\{g\in\cG: |g|=n\}\le d$ for all $n\in\bN$. Then $h_{\rm con}(X)=\log\lambda$, where $\lambda$ is the solution of $$\sum_{g\in\cG}\lambda^{-|g|}=1.$$
\end{lemma}
\begin{proof}
   For $n\in\bN$ denote by $\cG(n)$ the set of all generating words of length $n$, $\cG(n)=\{g\in\cG : |g|=n\}$ and let $c(n)= {\rm card}\,\cG(n)$. By our assumptions these cardinalities are bounded by some constant $d$, so we adopt $d=\max\{c(n): n\in \bN\}$. We define the gap-size sets $S_i$ as 
   \begin{equation*}
     S_i =\{n\in\bN: c(n+1) \geq i\}\text{ for } i=1,...,d.
   \end{equation*}
       Suppose that the generating set of $X$ has $r$ words of length one. We construct an isomorphism between $X$ and a coded shift on the alphabet $\{0,...,d+r\}$ by identifying generating words of length $n\ge 2$ with $0(i)^{n-1}$ where $1\le i \le d$ and words of length 1 with $i$ where $d+1\le i\le d+r$. Hence, we consider the coded shift $Y$ generated by $$\{0i^{s_i}: s_i\in S_i,\, i=1,...,d\}\cup\{i:i=d+1,...,d+r\},$$ which is a slight modification of the generalized $S$-gap shift discussed above. Application of the Pavlov's result \cite[Theorem 1.3]{Pavlov} tells us that 
       \begin{equation}\label{eq:ent_Y}
         \sum_{n=1}^{\infty}c(n)e^{-nh_{\rm top}(Y)}=1.
       \end{equation}
       
   Note that the residual set $Y_\res$ consists of bi-infinite strings ending or beginning with an infinite
   string of a single symbol $i\in\{1,...,d\}$. Any ergodic shift-invariant measure $\mu$ with $\mu(Y_\res)\ne 0$ is supported on the periodic points of the form $(...iii...)$.  This can be shown by a standard argument of partitioning the set of points ending with a string of $i$ into subsets $Y_N=\big{\{}y\in Y: y[N,\infty)=0iii... \text{ for some }i\in \{1,...,d\}\big{\}}$. Since the sets $Y_N$ are disjoint and each $Y_N=\sigma^{-N}(Y_0)$, we see that $\mu(Y_N)=0$. Similarly, the set of points beginning with a string of $i$ must have $\mu$-measure zero as well. Therefore, for any invariant probability measure $\mu$ on $Y$ with $h_\sigma(\mu)>0$ we have $\mu(Y_\con)=1$.

   Next we show that the topological entropy of $Y$ equals the concatenation entropy of $X$ by establishing an isomorphism between $X_{\rm con}$ and $Y_\con$. We enumerate the words in $\cG(n)$ so that $\cG(n)=\{g_{n,1},...,g_{n,c(n)}\}$ and write $\tau(g_{n,i})=i$ for $n>1$ and $\tau(g_{1,i})=d+i$. Hence, for a generating word $g$ its index in the set $\cG(|g|)$ 
   is given by $\tau(g)$. To define the map $T:X_{\rm con}\to Y_\con$ we represent each element of $X_{\rm con}$ as the concatenation of the words from the generator and then replace each generating word of length 1 by $d+i$ 
   and each generating word of length greater than 1 by a block of the form $0i^{(n-1)}$, where $n$ is the length of the word and $i$ is its index in $\cG(n)$. Precisely, for every $x\in X_{\rm con}$ there is a sequence of integers $(k_j)_{j\in \bZ}$ such that $x[k_j,k_{j+1})\in \cG$ for all $j$. Let
    $T(x)=y$ be so, that for $j\in\bZ$ 
    $$y[k_j,k_{j+1})=\begin{cases}
                        \tau\big(x[k_j,k_{j+1})\big), & \mbox{if $k_{j+1}-k_j=1$} \\
                        0\tau\big(x[k_j,k_{j+1})\big)^{k_{j+1}-k_j-1}, & \mbox{otherwise}.
                     \end{cases} $$

    Since $\cG$ uniquely represents $X_{\con}$, the sequence $(k_j)$ is unambiguously determined by $x\in X_\con$ and the map $T$ is well defined. It follows that $T$ is one-to-one from the definition of the index map $\tau$ and that $T$ is onto from the construction of the gap-sets $S_i$. Moreover, $T(\sigma x)=\sigma T(x)$. Let $\mu_Y$ be a measure of maximal entropy of $Y$ and consider the pull-back measure $\mu_Y \circ T$. Then $\mu_Y(Y_\con)=1$
    and the systems $(X,\sigma,\mu_Y\circ T))$ and $(Y,\sigma,\mu_Y)$ are measure-theoretically isomorphic, hence the entropy is preserved. It follows from the fact that $\mu_Y\circ T(X_\con)=1$ that the concatenation entropy of $X$ is at least the topological entropy of $Y$.
    On the other hand, if there is an invariant measure $\nu$ on $X$ satisfying $\nu(X_\con)=1$ and $h_\sigma(\nu)>h_{\rm top}(Y)$ then its push-forward $\nu\circ T^{-1}$ would be an invariant measure on $Y$ with entropy exceeding $h_{\rm top}(Y)$, which is a contradiction. We conclude that $h_\con(X)=h_{\rm top}(Y)$.

   We now apply formula (\ref{eq:ent_Y}) for the entropy of $Y$ and conclude that $h_\con(X)=\log\lambda$ where $\lambda$ must satisfy
   $$\sum_{g\in\cG}\lambda^{-|g|}=1.$$
   \end{proof}


Suppose a coded shift $X$ satisfies the assumptions of Lemma \ref{lem:entropy_formula} and has an infinite generating set $\cG$. We enumerate the elements of $\cG$, so that $\cG=\{g_i: i\in \bN\}$ and consider coded shifts $X_m$ which are generated by $\{g_1,...,g_m\}$. Since no elements have to be added in the closure of the concatenation part of $X_m$ these shifts are often easier to control than $X$. It follows from the previous lemma that the topological entropy of $X_m$ and the concatenation entropy of $X$ are given by $\log \lambda_m$ and $\log \lambda_\ast$ respectively, where $\lambda_m$ and $\lambda_\ast$ satisfy

$$\sum_{i=1}^{m}\lambda_m^{-|g_i|}=1\text{ and } \sum_{i=1}^{\infty}\lambda_\ast^{-|g_i|}=1.$$
We conclude that $(\lambda_m)_m$ is an increasing sequence with $\lim_{m\to\infty}\lambda_m=\lambda_\ast.$ Thus, we obtain the following:

\begin{corollary}\label{cor1}
  Let $X$ be a coded shift with generating set $\cG=\{g_i: i\in \bN\}$ which uniquely represents $X_\con$ and has a bounded rate of growth. Then $$\lim\limits_{m\to\infty} h_{\rm top}(X_m)=h_{\rm con}(X),$$ where $X_m$ are the coded shifts generated by $\{g_1,...,g_m\}$.
\end{corollary}

We are now ready to prove  Theorem \ref{thm:construstion}.

\begin{proof}[Proof of Theorem \ref{thm:construstion}]
We fix a word $a$ which is not  in $\cL(Z)$, but every proper subword of $a$ is in $\cL(Z)$.  Such a word exists and can be chosen to have length at least two since $Z$ is not a full shift.  Denote by $s$ the first symbol of $a$ and let $\tilde a$ be such that $s\tilde a=a$. Since $Z$ is irreducible, for each $n$ we can find a word $w_n\in\cL(Z)$ which contains all words in $\cL_n(Z)$ as subwords. Furthermore, we inductively select a sequence of words $(g_n)_{n\in\bN}$ in $\cL(Z)$ with the following properties:
\begin{itemize}
  \item[(i)] $g_n$ starts with $\tilde{a}$ and ends with $s$;
  \item[(ii)] $g_n$ contains $w_n$ as a subword;
  \item[(iii)] $|g_{n+1}|>|g_{n}|$.
\end{itemize}
We let $n_k=k\lceil \frac{1}{\epsilon}\rceil$, where $\lceil .\rceil$ is the ceiling function, and define the generating set $\cG=\{g_{n_k}: k\in \bN\}$. We note that every generator $g_{n_k}$ starts with the word $\tilde a$ and ends with the symbol $s$. Since $a=s\tilde a\notin\cL(Z)$ it follows that no bi-infinite concatenation of generators in $\cG$ belongs to $Z$. Hence, $X_{\rm con}(\cG)\cap Z=\emptyset$. On the other hand, for any $z\in Z$ and any $i\in \bN$,  $z[-i,i]$ is a subword of $g_{n_k}$ whenever $n_k> 2i$  and hence $z\in X(\cG)$.
It follows that $Z\subset X_{\rm res}(\cG)$. In fact, $Z\subset X_{\rm lim}(\cG)$, where $X_{\rm lim}(\cG)$ is defined in (\ref{def:X_lim}). Since $\cG\subset \cL(Z)$ the opposite inclusion holds as well, so that $Z=X_{\lim}$. Recall that for any $\mu\in\cM_\sigma(X)$ we have $\mu(X_{\con}\cup X_{\lim} )=1$. If $\mu(X_{\res})=1$ then $\mu(Z)=\mu(X_{\lim})=1$.

It follows from the construction that $\cG$ uniquely represents $X_{\con}$ since the appearances of the word $a$ within $x\in X_{\rm con}$ determine unambiguously the partition of $x$ into a sequence of words from $\cG$. 
To see this we observe that dependent on the structure of the word $a$ it is possible that there is a word $v$ in $x$ which is given by finitely many overlapping copies of $a$. However, by construction of the generators $g_n$ the transition between generators must in this case occur in the copy of $a$ on the right-hand side of the word $v$.
Hence we have established the assertions (i) and (ii).

To verify (iii) we note that the generating words of $X=X(\cG)$ have different lengths we can apply Lemma \ref{lem:entropy_formula} with $d=1$. We obtain that $h_\con(X)=\log \lambda_*$ where $\lambda_*$ satisfies the equation
$$\sum_{k=1}^{\infty}\lambda^{-|g_{n_k}|}=1.$$
To complete the proof we show that any number $\lambda$ for which $\log\lambda\ge \epsilon$ cannot be a solution to the above equation. Recall from our construction that $g_{n_k}$ includes a subword $w_{n_k}$ and hence $|g_{n_k}|>n_k\geq k/\epsilon$ by the choice of the sequence $(n_k)$.
Then $\log\lambda\ge\epsilon$ implies that $\lambda^{-|g_{n_k}|}< e^{-k}$, 
so that $\sum_{i=1}^{\infty}\lambda^{-|g_i|}<1.$ We conclude that $h_{\rm con}(X)=\log\lambda_*<\epsilon$, which completes the proof of the theorem.
\end{proof}

We now present a local structure result for the concatenation and residual entropy of coded shift spaces.
\begin{theorem}\label{thm:local_structure} 
Let $X=X(\cG)$ be a coded shift such that
$\cG$ uniquely represents $X_{\con}$ and $h_{\con}(X)>h_{\res}(X)$.
Then for every $\epsilon>0$ there exists a generating set $\widetilde{\cG}=\cG\cup \{f_i: i\in\bN\}$, where the $f_i$'s are finite words with $|f_i|<|f_{i+1}|$ for $i\geq 1$, such that the coded shift $\widetilde X = X(\widetilde \cG)$ has the following properties:
\begin{enumerate}
\item[(i)]   $\widetilde \cG$ uniquely represents $\widetilde X_{\con}$;
\item[(ii)] $h_{\res}(X)=h_{\res}(\widetilde X)$;
\item[(iii)] $h_{\rm top}(X)<h(\widetilde X_{\rm con})=h_{\top}(\widetilde X)<h_{\top}(X)+\epsilon$;
\item[(iv)] $\widetilde{X} \subset \widetilde X_{\lim} $.
\end{enumerate}
Moreover, it can be arranged for any given sequence $(m_i)_{i\in \bN}$ of positive integers  that $|f_i|\geq m_i$.
\end{theorem}

\begin{proof} 

Let $\cG=\{g_i: i\in \bN\}$ be an enumeration of $\cG$. We write $\cG(m)=\{g_1,\dots,g_m\}$. Let $X_m=X(\cG(m))$. Consider 
\[f_m(\lambda)=\sum_{i=1}^{m}\lambda^{-|g_i|},\]
which are strictly positive, decreasing and analytic functions on $(1,\infty)$ satisfying $\lim_{\lambda\to \infty}f_m(\lambda)=0$. 
Let $\lambda_m$ be the unique solution of $f_m(\lambda)=1$ and let $\lambda_*=\lim_{m\to\infty} \lambda_m$. Clearly $\cG(m)$ has a bounded rate of growth
 and thus Lemma \ref{lem:entropy_formula} implies $h_{\rm top}(X_m)=\log \lambda_m$. Moreover, since $h_\con(X)> h_\res(X)$ we may apply \cite[Proposition 28]{BDWY} and obtain 
\begin{equation}
    \lim_{m\to\infty} h_{\rm top}(X_m)= \lim_{m\to\infty} \log \lambda_m=\log \lambda_\ast= h_{\rm top}(X).
    \end{equation}
Let $\epsilon>0$ be given. We note that $(f_m')_m$ is a sequence of strictly negative  continuous functions on $(1,\infty)$ which is decreasing in $m$. Hence, there exists $c>0$ such that $f_m'(\lambda)<-c$ for all $\lambda\in [\lambda_\ast,\lambda_\ast+1]$ and all $m\in \bN$.
Putting these facts together shows that for every $\delta>0$ there exists $\tau>0$ such that
$f_m(\lambda_*+\delta)<1-\tau $ holds for all $m\in \bN$.
Let $\{m_i\}_{i \in \bN}$ be a strictly increasing sequence of positive integers. By making the $m_i$ larger if necessary we can assure that $\sum^{\infty}_{i=1}(\lambda_*+\delta)^{-m_i}<\tau$.
Our goal  is to define $\widetilde{\cG}=\cG\cup \{f_i\}$ with $|f_i|=m_i$. 
The additional generators $f_i$ will be defined below.
Let $\widetilde{\cG}=\{\widetilde{g}_i: i\in \bN\}$ be an enumeration of $\widetilde{\cG}$ and let
$\widetilde{f}_m,\widetilde{X}_m, \widetilde{\lambda}_m$ and $\widetilde{\lambda}_\ast$ be defined analogously as for $X$. It follows from an elementary calculus argument that $\widetilde{f}_m({\lambda_\ast}+\delta)\leq 1$  for all $m\in \bN$.
Hence $\widetilde{\lambda}_\ast\leq \lambda_\ast+\delta$. By making $\delta$ smaller if necessary and by further increasing the integers $m_i$, we can assure that
\begin{equation}\label{h+e}
   \lim_{m\to\infty} \log(\widetilde{\lambda}_m) = \log (\widetilde{\lambda}_{\ast})< \log ( \lambda_{\ast})+\epsilon .
\end{equation}

Next we construct the additional generators $f_i$.
Let $u,v$ be distinct words with $|u|=|v|$ which do not appear in the language of $X$. Moreover, we assume that the periodic points $p_u=u^{\infty}_\infty,p_v=v^\infty_{\infty}$ have disjoint orbits $O(p_u)$ and $O(p_v)$ with prime period $|u|$. Hence, $\{p_u,p_v\}\cap X=\emptyset$.
We note that the prime period condition can be guaranteed by  selecting $|u|$ to be a prime number. In fact, for any prime number $n$ satisfying $|\cA|^n-|\cL_n(X)|\geq 2n$ one can find such words $u$ and $v$ of length $n$. The existence of  prime numbers $n$ with this property follows from $h_{\rm top}(X)<\log |\cA|$ which is a consequence of the entropy-minimality of the full shift. 
Next we inductively define a sequence of words $(w(i))_{i\in \bN}$ and generators
\begin{equation}\label{deffi}
f_i=u^{2^i|w(i)|}w(i)v^{2^i|w(i)|}.
\end{equation}
Fix $g\in \cG$ with $|g|\geq 3$ and define $w(1)=g$ and $f_1$ by \eqref{deffi}. For $i\geq 2$
define $\cG_i=\cG \cup \{f_1,\dots,f_{i-1}\}$. We chose $w(i)$ to be a finite concatenation
of generators in $\cG_i$ which contains all words of length $i$ in $X(\cG_i)$

Moreover, we assume that the first and the last generator in $w(i)$ belong to $\cG$.
We also assume that $w(i)$ satisfies the following:
\begin{itemize}
\item[ (i)] $|w(i)| \geq m_i$;
\item[(ii)] $|w(i+1)|> |w(i)|$ for all $i \in \bN$;
\item[(iii)] $w(i)\subset w(i+1)$ for all $i \in \bN$.  
\end{itemize}
We now define $\widetilde \cG =\cG\cup \{f_i: i\in \bN\}$ and $\widetilde{X} = X(\widetilde \cG)$.
It follows from property $(ii)$ and Equation \eqref{deffi} that $|f_i| < |f_{i+1}|$ for all $i\geq 1$. 
We introduce some notation. We say $x\in \widetilde{X}$ contains a finite chain of generators $f_{i_1}\subset f_{i_2}\subset \dots \subset f_{i_l}$ if there exist
$r_l<r_{l-1}<\dots <r_1 < s_1<s_2<\dots < s_l$ such that $x[r_j,s_j]=f_{i_j}$ for all $j=1,\dots ,l$. Moreover, we say a finite chain of generators $f_{i_1}\subset f_{i_2}\subset \dots \subset f_{i_l}$ of $x$ is maximal if the chain cannot be extended to a larger finite chain of generators of $x$.
Analogously, we say $x$ contains an infinite chain of generators $f_{i_1}\subset f_{i_2}\subset \cdots \subset f_{i_l}\subset f_{i_{l+1}} \subset\cdots$ if there exist
$\cdots < r_{l+1}<r_l< \dots <r_1 < s_1<\dots < s_l<s_{l+1}<\cdots$ such that $x[r_j,s_j]=f_{i_j}$ for all $j\in \bN$.
 It follows readily from the definition of the generators $f_i$ that $x\in \widetilde{X}_{\con}$ cannot have infinite chains of generators in $\{f_i: i\in \bN\}$. Moreover, if $f_{i_1}\subset \cdots \subset f_{i_l}$ is a maximal finite chain of generators of $x\in \widetilde{X}_{\con}$  then 
$f_{i_l}=x[r_{i_l},s_{i_l}]$ is the only possibility in the representation of $x$ as a concatenation of generators in $\widetilde{\cG}$. 
To see this we notice that by \eqref{deffi} the  precise ${2^{i_l}|w({i_l})|}$-times concatenation of $u$ occurs at the beginning of $f_{i_l}$. Moreover, this representation is unique since by maximality of the chain, the generator $f_{i_l}$ cannot be a subword of a larger generator $f_{i_{l+1}}$.
Here we also use the fact that $p_u$ and $p_v$ are of prime period $|u|$.
Putting these properties  together implies that
$\widetilde \cG$ uniquely represents
$\widetilde X_{\con}$.  
Further, if $x\in \widetilde X_{\rm con}$ and $n\in \bN$, then there exists $N\in \bN$ such that $x[-n,n]\subset f_i$ for all $i\geq N$ which implies $\widetilde{X}_{\rm con} \subset \widetilde X_{\lim}$. In fact, we have proved that $\widetilde{X} \subset \widetilde X_{\lim}$.

To complete the proof we have to show that $h_{\res}(X)=h_{\res}(\widetilde X)$. This is a consequence of  the following claim.

Claim: If $\mu$ is an ergodic invariant measure with $\mu(\widetilde{X}_{\res}\setminus X_{\res})=1$ then $\mu=\mu_{p_u}$ or $\mu=\mu_{p_v}$.


Let $\mu$ be an ergodic invariant measure with $\mu(\widetilde{X}_{\res}\setminus X_{\res})=1$.
This means that $\mu(\widetilde{X}_{\res}\setminus X)=1$. Every element of the set $\widetilde{X}_{\res}\setminus X$ has either $u$ or $v$ as a subword. Then we take $Y=\tilde X_{\lim}\cap(\widetilde{X}_{\res}\setminus X)$ and conclude from Lemma 4.1 in \cite{Pavlov}  that $\mu(Y)=1$.
We recall that $$\widetilde X_{\text{lim}}=\{x \in \widetilde X:\  \forall n \in \bN\,\, \exists \tilde g(n) \in \widetilde \cG,\,\,  x[-n,n] \subset \widetilde g(n)  \}.$$

We first consider the set of points in $Y$ which contain an infinite (one-sided or two-sided) concatenation of $u$'s or $v$'s. We have the following possibilities:
\begin{itemize}
\item[(i)] $O(p_u)$ and $O(p_v)$; 
\item[(ii)] $\{u_\infty g_{k_1}g_{k_2}g_{k_3}\dots :\, g_{k_j}\in \widetilde{\cG}\}$ and $\{v_\infty g_{k_1}g_{k_2}g_{k_3}\dots :\, g_{k_j}\in \widetilde{\cG}\}$;

\item[(iii)] $\{\dots g_{k_{-3}}g_{k_{-2}}g_{k_{-1}} v^\infty: g_{k_j}\in \widetilde{\cG}\}$ and  $\{\dots g_{k_{-3}}g_{k_{-2}}g_{k_{-1}} u^\infty: g_{k_j}\in \widetilde{\cG}\}$.
\end{itemize}
We note that the sets in (i)-(iii) are shift invariant. If $\mu$ puts full measure on one of the periodic orbits in (i) then there is nothing to prove. Moreover, it follows from standard arguments that $\mu$ cannot put positive measures on either of the sets in (ii) and (iii). It remains to consider the set of points $x\in Y$ which only contain finite concatenations of $u$ and $v$. 
 We observe that $\mu$ puts zero measure on the set of points $x$ which only contain finite chains of generators in $\{f_i:i\in \bN\}$.
 This can be seen by evaluating the following cases: First, the set of points $x$ that have infinitely many maximal finite chains of generators in  $\{f_i:i\in \bN\}$ expanding in a negative and positive direction cannot have positive $\mu$-measure since these points belong to $\widetilde{X}_{\con}$. Here we also use the fact that by construction of the generators $f_i$, these finite maximal chains of generators can not overlap.
 Second, the sets of points $x$ which have infinitely many maximal finite chains of generators in  $\{f_i:i\in \bN\}$ which expand either in positive or in negative direction cannot carry positive $\mu$-measure. This follows from a standard argument applying the $\sigma$-invariance of $\mu$. Finally, the set of points $x$ which only have finitely many maximal finite chains of generators can also not carry positive $\mu$-measure.

It remains to consider the set $Y_{F}(\infty)$ of points $x$ which contain an infinite chain of generators in $\{f_i:i\in \bN\}$. Note that $Y_{F}(\infty)$ is shift invariant. Suppose that $\mu$ puts  positive (and hence full) measure on $Y_{F}(\infty)$. It follows from Birkhoff's Ergodic Theorem that the set of $\mu$-generic points in $Y_{F}(\infty)$ is of full $\mu$ measure. Let $x\in Y_{F}(\infty)$ be a $\mu$-generic point, that is,
\begin{equation}\label{eq222}
\mu_x(n)=\frac{1}{n}\sum_{k=0}^{n-1}\delta_{\sigma^k(x)}
\to \mu\,\,\,\,\,\, \mbox{for}\,\,\, n\to\infty,
\end{equation}
where the limit is taken in the weak$^\ast$ topology. We will show that \eqref{eq222} is not possible.
It follows from the definition of having an infinite chain of generators in $\{f_i:i\in \bN\}$ that for sufficiently large $l\in \bN$, the point $x$ is of the form 
\begin{equation}\label{eq222trick}
x=\cdots u^{2^{i_l}|w(i_l)|}x_{-r}x_{-r+1}\cdots .x_0x_1x_2\cdots x_{|w(i_l)|-r-1}
v^{2^{i_l}|w(i_l)|}\cdots
\end{equation}
for some $r=r(l)>0$. In particular $r<|w(i_l)|$. Let $\phi\in C(\widetilde{X},\bR), N\in\bN $ and $\delta>0$. Let $C=\sup\{|\phi(y)-\phi(\tilde y)|: y,\tilde y \in \widetilde{X}\}<\infty$. We now select $i_l$ as in Equation \eqref{eq222trick} with the following additional properties: We require that
\[
|\phi(y)-\phi(\tilde y)|<\frac{\delta}{2}\,\,\,\, \mbox{whenever}\,\,\, y[-s,s]=\tilde y[-s,s],
\]
where $s=|v|\, |w(i_l)|$. We further require that $n=|v|\, |w(i_l)|(2^i-1)+|w(i_l)|-r> N$ and that the following inequality holds:
\[
\frac{1}{n}\left((|v|\, |w(i_l)|+|w(i_l)|-r)C  + |v|\,|w(i_l)|(2^{i_l}-2)\frac{\delta}{2}\right)
<\delta.
\]
Let $t=|v|\,|w(i_l)|+|w(i_l)|-r$.
Then
\begin{align*}
\left|\int \phi\,d\mu_x(n)-\int \phi\,d\mu_{p_v}\right|&= \frac{1}{n}\left|\sum_{k=0}^{n-1} \phi(\sigma^k(x))-n\int \phi\,d\mu_{p_v}\right|\\
&\leq \frac{1}{n}\left|\sum_{k=0}^{t-1} \phi(\sigma^k(x))- t\int \phi\,d\mu_{p_v}\right|\\
&\hspace{1cm}+\frac{1}{n}\left|\sum_{k=t}^{n-1}\phi(\sigma^k(x))- (n-t)\int \phi\,d\mu_{p_v}\right|\\
&\leq  \frac{C}{n}(|v|\, |w(i_l)|+|w(i_l)|-r)\\
&\hspace{1cm}+ \frac1n|v|\,|w(i_l)|(2^{i_l}-2)\frac{\delta}{2}\\
&<\delta.
\end{align*}
Since $\phi,N$ and $\delta$ were arbitrary we conclude that $\mu_x(n)$ can not converge to a measure which puts full measure on $Y_F(\infty)$. This proves the claim which implies $h_{\res}(X)=h_{\res}(\widetilde X)$. We conclude that $h_{\res}(\widetilde X)<h_{\rm con}(\widetilde{X})$. 
By applying \cite[Proposition 28]{BDWY} to $\widetilde{X}$  we deduce that
 $\log ( \widetilde{\lambda_{\ast}})=h_{\top}(\widetilde X)$. Finally, we conclude from equation \eqref{h+e} that $h_{\top}(X) < h_{\top}(\widetilde X)< h_{\top}(X)+\epsilon.$
\end{proof}

\section{Topological Structure Arising from Coding}\label{sec:G-topology}

Let $X$ be a coded shift space with a generating set $\cG$. For $g\in\cG$ we define
\begin{align*}
  [g] & =\{x\in X: x=\cdots x_{-1}.x_0x_1\cdots\text{ with } x_i\in\cA\text{ and }x_0x_1...x_{|g|-1}=g\}; \\
  \llb g \rrb & =\{x\in X: x=\cdots g_{-2}g_{-1}.g_0g_1g_2\cdots\text{ with } g_i\in\cG\text{ and }g_0=g\}.
\end{align*}
Hence, $[g]$ is the standard cylinder set of the word $g$ in $X$, whereas $ \llb g \rrb$ is a set of all points in $X_\con$ which have
a representation as a concatenation of the elements of the generating set in which $g$ appears starting at the zeroth coordinate. Note that $ \llb g \rrb$ depends on a particular choice of the generating set $\cG$ for $X$, therefore we call the sets of this type $\cG$-cylinders. Assume now that $\cG$ uniquely represents $X_{\rm con}$. We define
\begin{equation}\label{defE}
E=E(\cG)=\bigcup_{g \in \cG} \llb  g \rrb.
\end{equation}
It follows from \eqref{defE} and the shift invariance of $X_{\rm con}$  that 
\begin{equation}\label{eqneu1}
X_{\rm con}=\bigcup_{n\in \bZ}\sigma^n(E),
\end{equation}
however the sets on the right-hand side of \eqref{eqneu1} are in general not disjoint.
The next lemma provides a decomposition of $X_\con$ into countably many disjoint sets.
\begin{lemma}\label{lem:partition} Let $X$ be a coded shift with the generating set $\cG$. If $\cG$ uniquely represents $X_\con$ then
  \begin{equation}\label{eqxsecE}
  X_\con = \bigcup_{g \in \cG} \bigcup^{|g|-1}_{k=0} \sigma^k \llb  g \rrb,
  \end{equation}
and the sets in this decomposition are pairwise disjoint.
\end{lemma}
\begin{proof}
Since $\cG$ uniquely represents the elements in $X_{\con}(\cG)$ the identity \eqref{eqxsecE} follows from the definitions of $\cG$-cylinders, the set $E$ and $X_{\rm con}$. Moreover, we have
$\llb g \rrb \cap  \llb h \rrb = \emptyset$ for $g \neq h $, $g,h \in \cG$ and in case $|g|>1$ also
$\sigma^k \llb g \rrb \cap \llb h \rrb = \emptyset$ for $1 \leq k \leq |g|-1$ and any $h \in \cG$.
This follows because $\sigma^k \llb g \rrb \cap E =\emptyset$ for $1 \leq k \leq |g|-1$ and $\llb h \rrb \subset E$
so that $\sigma^k \llb g \rrb \cap \llb h \rrb = \emptyset$.
The cases $\sigma^k \llb g \rrb \cap \sigma^l\llb h \rrb = \emptyset$ when
$1 \leq k \leq |g|-1$ and $1 \leq l \leq |h|-1$ (and $k \neq l$ if $g=h$)
are reduced to the previous case by applying $\sigma^{-\min(k,l)}$.
\end{proof}
We need an alternative representation of the set $E.$
We fix an enumeration of the elements of the generating set $\cG$, i.e., $\cG=\{g_i: i\in\bN\}$. Then we can associate with $X(\cG)$ the full shift on a countable alphabet $(Y,\sigma)=(\bN^{\bZ},\sigma)$ in the following way. Consider the map $\pi:Y\to X_\con$ defined by
\begin{equation}\label{eq:def_pi}
  \pi(y)=\cdots g_{y_{-2}}g_{y_{-1}}.g_{y_{0}}g_{y_{1}}g_{y_{2}}\cdots. 
\end{equation}
and let $E=\pi(Y)$. This definition of $E$ is consistent with \eqref{defE}. Hence, $E$ is the set of all concatenations of words from $\cG$, whose entry at zero is the first letter of a generator in $\cG$. We note that $\pi$ is a continuous injection.
Further note that both $X$ and $Y$ endowed with respective metrics defined in \eqref{defmetX} are  Polish spaces. 

We now show that the concatenation set, the residual set, $\cG$-cylinders as well as the set $E$ are Borel sets.

\begin{lemma}\label{gcylindermeas}
Let $X$ be a coded shift with generating set $\cG=\{g_i: i\in\bN\}$ such that $\cG$ uniquely represents $X_{\con}$, and let $g\in\cG$. Then the sets $E, X_{\rm con}, X_{\rm res}$ and $\llb g \rrb$ are Borel sets.
\end{lemma}
\begin{proof}
 Recall that $Y$ is a Polish space and  $\pi$ is a continuous injection. Therefore, $E=\pi(Y)$ being  Borel is a consequence of the Lusin-Suslin Theorem. Hence by \eqref{eqneu1}, $X_{\rm con}$ and $X_{\rm res}$ are Borel sets as well.
 It remains to consider the  $\cG$-cylinder  $\llb g \rrb$.
 Define $Y_g=\pi^{-1}(\llb g \rrb)=\{y=\dots y_{{-2}}y_{-1}.y_{0}y_{1}y_{2}\dots\in Y:  g_{y_0}=g\}$ and note that $Y_g=[y_0]$ is a Polish subspace of $Y$. 
 Applying again the Lusin-Suslin Theorem yields that $\pi(Y_g)=\llb g \rrb$ is a Borel set. 
 \end{proof}
We remark that the proof of Lemma \ref{gcylindermeas} does not generalize to the situation when $\cG$ does not uniquely represent $X_{\rm con}$. In this case the lack of injectivity of the map $\pi$ allows us only to conclude that the sets $E$ and $\llb g \rrb$ are analytic rather than Borel. However, one can still show that   $E, \llb g \rrb, X_{\rm con},X_{\rm 
 res}$ are universally measurable sets. In particular, these sets are $\mu$-measurable for all $\mu\in \cM_\sigma(X)$ provided one requires the measures in $\cM_\sigma(X)$ to be complete.
 Since we do not need this result we  leave the details to the reader.
 \\[0.3cm]
 Next we extend the notion of $\cG$-cylinders to finite concatenations of generators. Given $g_0,\dots,g_k\in\cG$ we define extended $\cG$-cylinders 
\[
\llb g_0\cdots g_k \rrb  =\{x=\cdots g'_{-1}.g'_0g'_1\cdots:\, g'_i\in\cG\text{ and }g'_j=g_j \text{ for } j=0,\dots,k\}.
\]
Clearly,
\[
\llb g_{0}\cdots g_{k}  \rrb=\bigcap_{i=0}^k \sigma^{-l_i}\llb g_{i}  \rrb,
\]
where $l_i=\sum_{j=0}^{i-1}|g_j|$. Thus, extended $\cG$-cylinders are Borel sets follows from Lemma \ref{gcylindermeas}.

For $w \in \cA^{\ast}$ we define $[w]_\con = [w] \cap X_\con$. Clearly $\llb g_{0}\cdots g_{k}  \rrb = \llb g_{0}\cdots g_{k}  \rrb \cap X_\con \subseteq [g_{0}\cdots g_{k}]_\con$.

\begin{lemma}\label{sigma_lim}Suppose $\cG$ uniquely represents $X_{\rm con}$.
For any $w \in \cA^{\ast}$ the set $[w]_\con  \subset X_\con$ is a countable disjoint union of shifted extended $\cG$-cylinders.
\end{lemma}
\begin{proof}
If $x \in [w]_\con$, then $x =\cdots g_{-1}g_0g_1\cdots$ with $g_{k} \in \cG$ and there is a word, say $g_{0}\cdots g_{j}$,
which contains $w=x[0,|w|-1]$ as a subword, so that $j \leq |w|$. There are only countably many such words,
and $x$ is contained in the shifted extended $\cG$-cylinder $\sigma^{l(x)}\llb g_{0}\cdots g_{k}  \rrb$ where
the shift $\sigma^{l(x)}$ places $g_{0}\cdots g_{k}$ so that the subword $w$
starts at position zero.
In fact, there may be several (but finitely many) possibilities that the word $w$ appears in $g_{0}\cdots g_{k}$. We conclude that  there
are finitely many respective shifts and a union of shifted extended $\cG$-cylinders satisfying
    
\begin{equation}\label{eqcoutG}
\bigcup^{n(w, g_{0}\cdots g_{k})}_{m =1} \sigma^{l_m}\llb g_{0}\cdots g_{k}  \rrb \subset [w]_\con. 
\end{equation}
It is easy to see that each of the shifted extended $\cG$-cylinders in the union is contained in $[w]_\con$.
By construction every $x \in [w]_\con$ is in one of the sets in \eqref{eqcoutG} and there are countably many of these sets.
We conclude that
$$[w]_\con=\bigcup_{w \subset g_{0}\cdots g_{k} \in \cG} \bigcup^{n(w, g_{0}\cdots g_{k})}_{m =1} \sigma^{l_m}\llb g_{0}\cdots g_{k}  \rrb $$
is a countable union of shifted extended $\cG$-cylinders. The disjointness of the cylinders in this union follows readily from similar arguments as in Lemma \ref{sigma_lim}.
\end{proof}

Recall that (shifted) $\cG$-cylinders and  (shifted) extended $\cG$-cylinders  are Borel subsets contained in $X_{\con}$.
We now define a class of shift invariant measures which assign full measure to $X_\con$. Namely,
we assign for $g \in \cG$ a measure $p_g \in [0,1]$ to the $\cG$-cylinder $\llb g \rrb$.
We also assume that $\sum_{g \in \cG}p_g=1$.
We then extend the probabilities independently to all $\cG$-cylinders by using shift invariance and intersections
of $\cG$-cylinders. By the generating property of the $\cG$-cylinders this gives a measure, say $m$,  on
$X_\con$, so that $m(\llb g_{0}\cdots g_{k}\rrb)=p_{g_0}\cdots p_{g_k}$.
 We note that the measure $m$ is not
necessarily finite. It follows from Lemma \ref{lem:partition} that for any shift invariant measure $\mu$ that assigns full measure to $X_{\con}$ we have
\begin{equation}\label{eq:mu(X)}
  \mu(X)=\sum_{g \in \cG} |g| \mu(\llb g \rrb).
\end{equation}
Therefore, in order for the measure $m$ to be finite, we must have $\sum_{g \in \cG} |g| p_g<\infty$. In this case we may normalize $m$ to obtain a probability measure on $X$. Because of the similarity with classical Bernoulli measures as well as its dependence on the coding $\cG$ we call the corresponding measures $\cG$-Bernoulli. This is formalized as follows.
\begin{definition}Assume $\cG$ uniquely represents $X_{\rm con}$.
   We say a measure $\mu\in \cM_\sigma(X(\cG))$ is  $\cG$-Bernoulli if there exist positive real numbers $p_g,\, g\in \cG$ with $\sum_{g \in \cG}p_g=1$ and  $c=\sum_{g \in \cG} |g| p_g<\infty$ such that for any $k\in\bN$ and any $g_0,...,g_k\in\cG$ we have
  $$\mu(\llb g_0\cdots g_k\rrb)=\frac{1}{c}p_{g_0}\cdots p_{g_k}.$$
\end{definition}

\section{Uniqueness of Measures of Maximal Entropy}\label{sec:mme}

In this section we establish the uniqueness of measures of maximal entropy for a large class of coded shift spaces. Note that the existence follows from expansivity of the system. If the generating set consists of finitely many words then the corresponding coded shift is sofic. Since the properties of equilibrium states on sofic shifts are very well understood, for simplicity of exposition we assume in the following that our generating set is countably infinite. All our results remain valid when the generating set is finite.

Let $X=X(\cG)$ be a coded shift such that $\cG$ uniquely represents $X_{\rm con}$.
We recall the definition of the map $\pi:Y\to E$ in \eqref{eq:def_pi}. 
As usual define the first return time to be $r_E(x)=\min \{n \geq 1: \ \sigma^n(x) \in E \}$  and the induced shift map $\sigma_E(x)=\sigma^{r_E(x)}(x)$. The fact that $\cG$ uniquely represents $X_\con$ implies that $\pi$ is invertible. In this case $\pi$ is a continuous bijection between $Y$ and $E$ and the following diagram commutes: 


\begin{equation}\label{eqYpiE}
\begin{tikzcd}
Y \arrow[r, "\sigma" ] \arrow[d, "\pi"] & Y  \arrow[d, "\pi"]  \\
E \arrow[r, "\sigma_E"] & E
\end{tikzcd}
\end{equation}

We will exploit \eqref{eqYpiE} to relate $\sigma_E$-invariant measures on $E$ to $\sigma$-invariant measures on $Y$.
Let $\cM_E$ be the set of $\sigma_E$-invariant probability measures on $E$ and let $$\cM_\con=\{\mu\in\cM_\sigma(X):\mu(X_\con)=1\}.$$ It follows from (\ref{eq:mu(X)}) that for any $\mu\in\cM_\con$ we have that $\mu(E)>0$. Hence, we can define a map $\bi:\cM_\con\to\cM_E$ by
\begin{equation}\label{eq:ind_map}
  \bi(\mu)=\frac{1}{\mu(E)}\mu|_E.
\end{equation}
We will show in the next lemma that the map $\bi$ is injective and that the image of $\cM_\con$ under $\bi$ is precisely the set
\begin{equation}\label{eq:induced_set}
  \cM_{\rm ind}=\{\mu\in\cM_E: \sum_{i=1}^{\infty}|g_i|\mu(\llb g_i \rrb)<\infty\}.
\end{equation}
We remark that the formula for a lift of a $\sigma_E$-invariant measure on $E$ to a $\sigma$-invariant measure on $X$ provided below could be deduced from the general theory for induced systems, see, e.g., \cite{aar}. However, the direct proof is fairly short and we include it for completeness.
\begin{lemma}\label{lem:lift}
  Suppose $\cG=\{g_i\}_{i\in\bN}$ uniquely represents $X_\con$. Then the map $\bi:\cM_\con\to\cM_{\rm ind}$ defined in \eqref{eq:ind_map} is a bijection and its inverse $\bl:\cM_{\rm ind}\to \cM_\con$ is given by
  $$\bl(\mu)(A)=\left(\sum_{i=1}^{\infty}|g_i|\mu(\llb g_i \rrb)\right)^{-1}\sum_{i=1}^{\infty}\sum_{j=0}^{|g_i|-1}\mu(\llb g_i \rrb\cap\sigma^{-j}(A)).$$
  \end{lemma}
\begin{proof}
  Note that measure $\tilde{\mu}$ given by
  $\tilde{\mu}(A)=\sum_{i=1}^{\infty}\sum_{j=0}^{|g_i|-1}\mu(\llb g_i \rrb\cap\sigma^{-j}(A))$ is a $\sigma$-invariant measure on $X$.
  Indeed,
  \begin{align}\label{eq:sigmainv}
  \begin{split}
    \tilde{\mu}\circ \sigma^{-1}(A)&= \sum_{i=1}^{\infty}\sum_{j=0}^{|g_i|-1}\mu(\llb g_i \rrb\cap\sigma^{-(j+1)}(A)) \\
    & = \sum_{i=1}^{\infty}\sum_{j=1}^{|g_i|-1}\mu(\llb g_i \rrb\cap\sigma^{-j}(A))+\sum_{i=1}^{\infty}\mu(\llb g_i \rrb\cap\sigma^{-|g_i|}(A))
  \end{split}
    \end{align}
  Since $\sigma^{|g_i|}=\sigma_E$ on $\llb g_i\rrb$ and $\mu$ is $\sigma_E$-invariant we obtain that $\mu(\llb g_i \rrb\cap\sigma^{-|g_i|}(A))=\mu(\sigma^{|g_i|}\llb g_i \rrb\cap(A))$. Recall that the sets $\sigma^{|g_i|}\llb g_i \rrb$ are pairwise disjoint, which gives a decomposition $E=\bigcup_{i=1}^\infty\sigma^{|g_i|}\llb g_i \rrb$. It follows that
  $$\sum_{i=1}^{\infty}\mu(\llb g_i \rrb\cap\sigma^{-|g_i|}(A))=\mu(E\cap A)=\sum_{i=1}^{\infty}\mu(\llb g_i \rrb\cap A).$$
  Replacing the last term of (\ref{eq:sigmainv}) by this expression and combining the sums we see that $\tilde{\mu}\circ\sigma^{-1}(A)=\tilde{\mu}(A)$ and hence $\tilde{\mu}$ is $\sigma$-invariant. When $\mu\in\cM_{\rm ind}$ the quantity $\tilde{\mu}(X)=\sum_{i=1}^{\infty}|g_i|\mu(\llb g_i \rrb)$ is finite so that we can normalize $\tilde{\mu}$ and obtain that $\bl(\mu)=\frac{\tilde{\mu}}{\tilde{\mu}(X)}\in\cM_{\rm con}$. 

  On the other hand, consider any $\nu\in\cM_\con$ and let $\mu=\bi(\nu)$. Then by (\ref{eq:mu(X)})
   $$\sum_{i=1}^{\infty}|g_i|\mu(\llb g_i \rrb)=\sum_{i=1}^{\infty}|g_i|\frac{\nu(\llb g_i \rrb)}{\nu(E)}=\frac{1}{\nu(E)}\nu(X)=\frac{1}{\nu(E)}.$$
   Using Lemma \ref{lem:partition} and $\sigma$-invariance of $\nu$ we see that $$\nu(A)=\sum_{i=1}^{\infty}\sum_{j=0}^{|g_i|-1}\mu(\sigma^{j}\llb g_i \rrb\cap A)=\sum_{i=1}^{\infty}\sum_{j=0}^{|g_i|-1}\mu(\llb g_i \rrb\cap\sigma^{-j}(A))$$
  for any measurable $A\subset X$. Substituting both expressions into the formula for the lift map $\bl$ we obtain
  $$\bl\circ\bi(\nu)(A)=\nu(E)\sum_{i=1}^{\infty}\sum_{j=0}^{|g_i|-1}\frac{\nu(\llb g_i \rrb\cap\sigma^{-j}(A))}{\nu(E)}=\sum_{i=1}^{\infty}\sum_{j=0}^{|g_i|-1}\mu(\sigma^{j}\llb g_i \rrb\cap A)=\nu(A).$$

\end{proof}

To deduce the existence and uniqueness of the measures which maximize the concatenation entropy we utilize the relationship \eqref{eqYpiE} between the systems $(E,\sigma_E)$ and $(Y,\sigma)$ and the ergodic theory for shifts with 
countable alphabet developed in \cite{MU,Sarig}. In \cite{MU} Mauldin and Urbanski proved that a H\"{o}lder potential $\psi$ defined on a one-sided shift with countable alphabet  $\bN^\bN$ has a unique equilibrium state (which is also Gibbs) whenever 
\begin{align}
  & \sum_{i=1}^\infty\exp(\sup \psi\vert_{[i]})<\infty\label{MU1}\\
  &\text{ and }\notag\\ 
  & \sum_{i=1}^\infty \inf(-\psi\vert_{[i]})\exp(\inf \psi\vert_{[i]})<\infty\label{MU2}
\end{align} 
hold, where $[i]$ denotes the one-sided cylinder of $i\in\bN$, i.e. $[i]$ is the set of all elements of $\bN^{\bN}$ with first coordinate $i$.
Here we say (following \cite{MU}) that a potential $\psi$ on a shift with a countable alphabet  is H\"{o}lder if
there are  constants $C$ and $\alpha$ such that for any $y'$ and $y''$ with $y_0'=y_0''$ we have
$|\psi(y')-\psi(y'')|\leq Cd(y',y'')^{\alpha}.$ 
Here $d(.,.)$ is the standard metric on $X$ defined in \eqref{defmetX}.
Note that the inequality must hold only when $y'_0=y''_0$, which differs from the case of a finite alphabet. Although our shift $Y$ is two-sided, our potential $\psi$ will not depend on the backward history (it is constant on cylinders of length one) and hence the results from \cite{MU} are applicable in our settings.

To estimate the convergence of the above series we use the representation of the concatenation set of a coded shift as an edge shift for a connected countable directed graph $\Gamma$. 
The nature of this representation as well as the way to adapt relevant results from graph theory to our settings is explained in \cite{Pavlov} (also see \cite{Petersen}). We do not repeat it here, but make use of the conclusions. Suppose $\Gamma$ is a countable labeled directed graph associated to the coded shift $X(\cG)$ and $h_G(\Gamma)$ is its Gurevich entropy. Then we always have $h(\cG)\le h_G(\Gamma)\le h_{\rm top}(X)$ and 
\begin{equation}\label{eq:Guevich}
 \sum_{n=1}^{\infty}c(n)e^{-nh_G(\Gamma)}\le 1,
\end{equation}
where $c(n)={\rm card}\{g_i\in\cG:|g_i|=n\}$ and $h(\cG)=\limsup\frac{1}{n}\log c(n)$ as before. 
Gurevich proved in \cite{gu} that 
the existence of a Borel probability measure on the edge shift for $\Gamma$ with entropy $h_G(\Gamma)$ is equivalent to
$\Gamma$ being positive recurrent. According to the Vere-Jones classification \cite{vj} for connected countable directed graphs this means (in our notation) that
\begin{equation}\label{eq:Vere-Jones}
  \sum_{n=1}^{\infty}c(n)e^{-nh_G(\Gamma)}=1\quad{\rm and}\quad \sum_{n=1}^{\infty}nc(n)e^{-nh_G(\Gamma)}<\infty.
\end{equation}

\begin{lemma}\label{lem:Gurevich}
  Let $X=X(\cG)$ be a coded shift such that $\cG=\{g_i: i\in \bN\}$ uniquely represents $X_{\con}$ and $h_{\con}(X)>h_{\res}(X)$. Then
  $$\sum_{i=1}^{\infty}e^{-|g_i|h_{\rm top}(X)}=1\quad{\rm and}\quad \sum_{i=1}^{\infty}|g_i|e^{-|g_i|h_{\rm top}(X)}<\infty.$$
\end{lemma}
\begin{proof} Consider $f(\lambda)= \sum_{n=1}^{\infty}c(n)\lambda^n$ with $c(n)={\rm card}\{g_i\in\cG:|g_i|=n\}$. Note that the radius of convergence of this power series is $e^{-h(\cG)}$ and hence $f$ is analytic on the interval $(-e^{-h(\cG)}, e^{-h(\cG)})$ as is its derivative $f'(\lambda)=\sum_{n=1}^{\infty}nc(n)\lambda^{n-1}$. 

If $h(\cG)<h_{\rm top}(X)$ then $\lambda_\ast=e^{-h_{\rm top}(X)} \in (-e^{-h(\cG))}, e^{-h(\cG))})$ and the series $\sum_{n=1}^{\infty}nc(n)e^{-nh_{\rm top}(X)}$ converges. Now consider coded shifts $X_m$ generated by $\{g\in \cG: |g|\le m\}$. It was proven in \cite[Proposition 28]{BDWY} that  whenever $h_\con(X)> h_\res(X)$ we have $\lim_{m\to\infty} h_{\rm top}(X_m)= h_{\rm top}(X)$. Since the shifts $X_m$ have finite generating sets their topological entropies satisfy

$$\sum_{n=1}^{m}c(n)e^{-nh_{\rm top}(X_m)}=1.$$ Letting $\lambda_m=e^{-h_{\rm top}(X_m)}$ we see from the above that $\lambda_m\to\lambda_\ast$ and the functions $f_m(\lambda)=\sum_{n=1}^mc(n)\lambda^n$ converge uniformly to $f(\lambda)$ in the neighborhood of $\lambda_\ast$. Since $f_m(\lambda_m)=1$, we must have $f(\lambda_\ast)=1$ as well. By the definitions of $f(\lambda)$ and $\lambda_\ast$ this means that $\sum_{n=1}^{\infty}c(n)e^{-nh_{\rm top}(X)}=1$.
If $h(\cG)=h_{\rm top}(X)$ then we must necessarily have $h_{G}(\Gamma)=h_{\rm top}(X)$. By our assumption $h_{\rm top}(X)=h_\con(X)$ and hence at least one of the measures of maximal entropy on $X$ assigns full measure to $X_\con$. Applying the correspondence between $X_\con$ and the edge shift of $\Gamma$, we see that the assumptions of Gurevich's theorem are satisfied. It now follows from (\ref{eq:Vere-Jones}) that $\sum_{n=1}^{\infty}c(n)e^{-nh_{\rm top}(X)}=1$ and $\sum_{n=1}^{\infty}nc(n)e^{-nh_{\rm top}(X)}<\infty$. 

To conclude the proof of the lemma we observe that $\sum_{i=1}^{\infty}e^{-|g_i|h_{\rm top}(X)}=\sum_{n=1}^{\infty}c(n)e^{-nh_{\rm top}(X)}$ and $\sum_{i=1}^{\infty}|g_i|e^{-|g_i|h_{\rm top}(X)}=\sum_{n=1}^{\infty}nc(n)e^{-nh_{\rm top}(X)}$.
\end{proof}
\begin{remark}
    As a consequence of Lemma \ref{lem:Gurevich}, (\ref{eq:Guevich}), and (\ref{eq:Vere-Jones}) we obtain that when $h_\con(X)>h_\res(X)$ then the topological entropy of $X$ coincides with the Gurevich entropy of the countable labeled graph associated to $X$, i.e. $$h_{\rm top}(X)=h_G(\Gamma).$$
\end{remark}

We are now ready to prove Theorem B which we restate here for convenience.

\begin{theorem}\label{thm:unique_mme}
Let $X=X(\cG)$ be a coded shift such that $\cG$ uniquely represents $X_{\con}$ and $h_{\con}(X)>h_{\res}(X)$.
Then $X$ has a unique measure of maximal entropy $\mu_{\max}$,  and $\mu_{\max}$  is
$\cG$-Bernoulli with $p_g=\exp(-|g|h_\top(X))$ and $c=\sum_{g\in \cG} |g| \exp(-|g|h_\top(X))$.
\end{theorem}
\begin{proof}
  Consider the constant potential $\phi: X \rightarrow \bR$ given by $\phi(x)=-h_\top(X)$,
so that $P_\top(\phi)=0$. Let $\mu$ be a measure of maximal entropy of $X$. Then $\mu$ is an equilibrium state of $\phi$ and, since $h_{\con}(X)>h_{\res}(X)$, we know that $\mu\in\cM_\con$ and $\bi(\mu)$ is well defined. The corresponding induced potential $\phi_E:E\to\bR$ has the form
$$\phi_{E}(x)=\sum_{j=0}^{r_E(x)-1}\phi\circ \sigma^j(x).$$
Note that in our notation (\ref{eq:def_pi}) the first return time $r_E(x)=|g_{y_0}|$ where $y=\pi^{-1}(x)$. In particular, when $x\in\llb g_i\rrb$ we have $r_E(x)=|g_i|$. Define the potential $\psi:Y\to\bR$ by $\psi=\phi_E\circ \pi$, Hence, $\psi(y)=-|g_{y_0}|h_{\rm top}(X)$. We first verify that if $\mu$ is a measure of maximal entropy for $X$, then $\nu=(\pi^{-1})_\ast \bi(\mu)$ is an equilibrium state for $\psi$. Since any invariant measure admits an ergodic decomposition we may assume that $\mu$ is ergodic and hence assigns full measure to $X_\con$.
It immediately follows from Lemma \ref{lem:partition} that
$$\int_X\phi\,d\mu=\sum_{i=1}^{\infty}\sum_{j=0}^{|g_i|-1}\int_{\sigma^j\llb g_i\rrb}\phi\,d\mu=\sum_{i=0}^{\infty}\int_{\llb g_i\rrb}\sum_{j=0}^{|g_i|-1}\phi\circ \sigma^j\,d\mu=\mu(E)\int_E\phi_E\,d\bi(\mu).$$
Since by Abramov's formula we also have $h_\sigma(\mu)=\mu(E)h_{\sigma_E}(\bi(\mu))$, using the properties of $\pi$  we obtain that
\begin{align*}
  h_\sigma(\nu)+\int_Y\psi\,d\nu & = h_{\sigma_E}(\bi(\mu))+\int_E\phi_E\,d\bi(\mu)\\
  & = \frac{1}{\mu(E)}\left[h_\sigma(\mu)+\int_X\phi\,d\mu\right] \\
  & =0.
\end{align*}
Hence, $P_{\rm top}(\psi)\ge 0$. To show the opposite inequality we use the result from \cite{MU} that $P_{\rm top}(\psi)=\sup_{m\in\bN} P_{\rm top}(\psi,Y_m)$, where $P_{\rm top}(\psi,Y_m)$ is the topological pressure of $\psi$ for the dynamical system $(Y_m,\sigma)$ with $Y_m=\{1,...,m\}^\bZ$. Denote by $\nu_m$ the unique equilibrium state of $\psi$ on $Y_m$ and let $\pi_\ast\nu_m$ be its pushforward measure on $E$. Then $$\sum_{i=1}^{\infty}|g_i|\pi_\ast\nu_m(\llb g_i\rrb)=\sum_{i=1}^{m}|g_i|\nu_{m}([i])<\infty,$$ which assures that $\pi_\ast(\nu_m)\in \cM_{\rm ind}$. By Lemma \ref{lem:lift} the measure $\mu_m=\bl(\pi_\ast\nu_m)$ is a $\sigma$-invariant probability measure on $X$ for which $h_{\sigma_E}(\pi_\ast\nu_m)=\frac{h_\sigma(\mu_m)}{\mu_m(E)}$ and $\pi_\ast \nu_m(\llb g_i\rrb)=\frac{\mu_m(\llb g_i\rrb)}{\mu_m(E)}$. Keeping in mind that $\psi(y)=-|g_{y_0}|h_{\rm top}(X)$ and that $\pi_\ast$ preserves the entropy of $\nu_m$ we evaluate
\begin{align*}
   P_{\rm top}(\psi|_{Y_m})& =h_\sigma(\nu_m)+\int_Y \psi\,d\nu_m \\
   & =  h_{\sigma}(\nu_m)-h_{\rm top}(X)\int_Y |g_{y_0}|\,d \nu_m \\
   & =  h_{\sigma}(\nu_m)-h_{\rm top}(X)\sum_{i=1}^{m}|g_i|\nu_m([i]) \\
  & =  h_{\sigma_E}(\pi_\ast\nu_m)-h_{\rm top}(X)\sum_{i=1}^{m}|g_i|\pi_\ast\nu_m(\llb g_i\rrb)\\
  & = \frac{1}{\mu_m(E)}\big(h_\sigma(\mu_m)-h_{\rm top}(X)\big),
\end{align*}
where in the last line we use the fact that $\sum_{i=1}^{m}|g_i|\mu_m(\llb g_i\rrb)=\mu_m(X)=1$.
Since $\mu_m$ is an invariant probability measure on $X$, the variational principle implies $h_\sigma(\mu_m)-h_{\rm top}(X)\le 0$. Taking the supremum over $m\in\bN$ we conclude that $P_{\rm top}(\psi)=0$.

To establish the uniqueness of the measure of maximal entropy it remains to be shown that for the potential $\psi(y)=-|g_{y_0}|h_{\rm top}(X)$ the series (\ref{MU1}) and (\ref{MU2}) converge. Since $\psi(y)=-|g_i|h_{\rm top}(X)$ for all $y\in [i]$, these series are 
$$ \sum_{i=1}^{\infty}e^{-|g_i|h_{\rm top}(X)}\quad{\rm and}\quad\sum_{i=1}^{\infty}|g_i|e^{-|g_i|h_{\rm top}(X)}.$$ 
The second sum clearly dominates the first. The proof that the second sum is finite is given in Lemma \ref{lem:Gurevich}. This shows that the measure of maximal entropy for $X$ is unique and we denote it by $\mu_{\rm max}$.

Next we show that $\mu_{\rm max}$ is $\cG$-Bernoulli. Recall that $h_\con(X)>h_\res(X)$. Therefore, Lemma \ref{lem:Gurevich} asserts that $h_{\rm top}(X)=\log \lambda_\ast$ where $\sum_{i=1}^{\infty}\lambda_\ast^{-|g_i|}=1$. Hence, $\sum_{i=1}^{\infty}\exp(-|g_i|h_{\rm top}(X))=1$ and we can consider the Bernoulli measure $\nu$ on $Y$ which assigns the mass $p_i=\exp(-|g_i|h_{\rm top}(X))$ to the cylinder $[i]$. It is straight forward to check that $\nu$ is an equilibrium state for $\psi$, and hence, by uniqueness, $\mu_{\rm max}$ coincides with the lift of $\pi_\ast\nu$. Applying the formula for the lift given in Lemma \ref{lem:lift} we see that $\mu_{\rm max}(\llb g_{i_1}...g_{i_k}\rrb )=\frac{1}{c}p_{i_1}\cdots p_{i_k}$ where $c=\sum_{i=1}^{\infty}|g_i|\pi_\ast\nu(\llb g_i\rrb)=\sum_{i=1}^{\infty}|g_i|p_i=\sum_{i=1}^{\infty}|g_i|\exp(-|g_i|h_{\rm top}(X))$.
\end{proof}

We remark that in the last part of the proof we actually show that conditions $\sum_{i=1}^\infty e^{-|g_i|h_{\rm top}(X)}=1$ and $\sum_{i=1}^\infty|g_i|e^{-|g_i|h_{\rm top}(X)}<\infty$ imply the existence of a measure of maximal entropy $\mu\in\cM_\con$. 
Moreover, $\mu$ is the only such measure in $\cM_\con$ and it is ergodic and $\cG$-Bernoulli. The condition $h_\con(X)>h_\res(X)$ is required to guarantee that there are no measures of maximal entropy which assign full measure to $X_\res$. We obtain

\begin{corollary}\label{cor:h_seq=h_res} Suppose $X$ is a coded shift with generating set $\cG$ which uniquely represents $X_\con$. Then the following are equivalent.
\begin{itemize}
  \item[(i)] There exists a measure of maximal entropy $\mu$ with $\mu(X_\con)=1$.
  \item[(ii)] $\sum_{g\in\cG}e^{-|g|h_{\rm top}(X)}=1$ and $\sum_{g\in\cG}|g|e^{-|g|h_{\rm top}(X)}<\infty$
\end{itemize}
In this case $\mu$ is unique in $\cM_\con$, and is $\cG$-Bernoulli with $p_g=e^{-|g|h_{\rm top}(X)}$ and $c=\sum_{g\in \cG} |g| e^{-|g|h_\top(X)}$.
\end{corollary}

This corollary supplies a characterization of the measure of maximal entropy from $\cM_\con$ in the case when $h_\con(X)=h_\res(X)$. Moreover, it could be used even in the case $h_\con(X)<h_\res(X)$ by employing another generating set specifically tailored to a particular measure of maximal entropy. We illustrate this technique on the Dyck shift in the Example \ref{ex:Dyck}.

Assume that $X=X(\cG)$ is uniquely represented and $h_{\con}(X)> h_{\res}(X)$. We have shown
that $X$ has a {\em unique} measure of maximal entropy $\mu_{\max}$. We now provide an
example which shows that this measure is, in general not a Gibbs measure.
\begin{example}\label{ex:notGibbs} There exists an intrinsically ergodic coded shift $X(\cG)$ whose measure of maximal entropy is $\cG$-Bernoulli but not Gibbs.
\end{example}

\begin{proof}
  We start with a strictly increasing sequence of positive integers $\{n_i\}_{i\in \bN}$, so that
$\lim_{i \to \infty} (n_{i+1}-n_i) =\infty$. Now let $\cG=\{g_i :\ g_i = 0^{n_i}1^{n_i}  \}$ and consider $X=X(\cG)$.
Clearly $\cG$ uniquely represents $X_{\con}$ and
\[ 0=h_{\res}(X)< h_{\con}(X)=h_{\top}(X)< \log 2.\]
We know from Theorem \ref{thm:unique_mme} that $X$ has a unique measure of maximal entropy $\mu_{\max}$ which is $\cG$-Bernoulli and consequently
\[\mu_{\max}\left(\llb g_{i} \rrb\right)=\frac{1}{c}\exp(-|g_{i}| h_{\top}(X))= \frac{1}{c}\exp(-2 n_i h_{\top}(X)), \]
with a normalizing constant $c$. Now for $i \in \bN$ consider the words $w_{i}=0^{n_i+1}1^{n_i+1}$,
which are subwords of every $g_{j}$ for all $j>i$. Moreover, there is a unique $0\le k_j<|g_j|$ such that $[w_i]\cap \sigma^{k_j}\llb g_j\rrb \ne \emptyset$. Since $\mu_{\rm max}$ assigns full measure to $X_\con$ by Lemma \ref{lem:partition} we obtain

\begin{align*}
  \frac{\mu_{\max}\left([w_{i}]\right)}{\exp(-2 (n_i +1)  h_{\top}(X))} & =
\frac{1}{\exp(-2 (n_i +1)  h_{\top}(X))}\sum_{j>i}\mu_{\max}\left(\llb g_{j}\rrb\right) \\
  & =\frac{1}{c} \exp(2 (n_i +1)  h_{\top}(X))\sum_{j>i}\exp(-2 n_j h_{\top}(X))\\
  & =\frac{1}{c} \sum_{j>i}\exp(-2 (n_j-n_i -1)  h_{\top}(X)) 
\end{align*}
Since by assumption $\{n_i \}$ is monotonically increasing with $(n_{i+1}-n_i) \rightarrow \infty$ as $i \rightarrow \infty$, the  last sum can be made as small as we wish. This violates the lower bound needed for the Gibbs property, hence
$\mu_{\max}$ is not Gibbs.
\end{proof}

The following example example establishes the existence of a coded shift $X=X(\cG)$ which has multiple measures of maximal entropy and satisfies $h_{\con}(X)=h_{\res}(X)$.
A similar example (with fewer properties) was given in
\cite{Pavlov}[Example 5.3]

\begin{example} \label{ex:hseq=hres} There exists a coded shift $X=X(\cG)$ such that $\cG$ uniquely represents $X_\con$, $h_{\con}(X)=h_{\res}(X)$, and $X$ has precisely three ergodic measures of maximal entropy. Two of these ergodic measures of maximal entropy are supported on disjoint proper  subshifts $Z_1,Z_2\subset X$. The third ergodic measure of maximal entropy is a $\cG$-Bernoulli measure on $X_{\rm con}$.
\end{example}
\begin{proof}
    Let $\cA=\{0,1,2,3,4\}.$ We will construct a coded shift $X$ over the alphabet $\cA$ which has the desired properties. 
For $n \in \bN$ we define $$W(n)=\{vw4^n:\ v \in \{0,1\}^n, w \in \{2,3\}^n\} \subset \cA^{3n}.$$ 
Further, we define $\cG=\bigcup_{n \in \bN}W(n)$ and $W_n=\bigcup^n_{k=1}W(k)$. Let $X=X(\cG)$ and $X_n=X(W_n)$. Clearly, $\cG$ uniquely represents $X_{\con}$.   

It follows from the definitions of $X_{\rm res}$ and $X_{\rm lim}$ that the disjoint shift spaces $Z_1=\{0,1\}^{\bZ}$, $Z_2=\{2,3\}^{\bZ}$, and $Z_3=\{4\}^{\bZ}$ are contained in $X_{\rm res}\cap X_{\lim}$. Let $x\in X_{\lim}\setminus (Z_1\cup Z_2\cup Z_3)$. Then $x$ must have precisely one 
transition index between rays in $Z_1,Z_2$ or $Z_2,Z_3$. More precisely, there is $i\in\{1,2\}$ and $k\in \bZ$ such that $x=x(-\infty,k)x[k,\infty)=uv$, where $u$ is a ray in $Z_i$ and $v$ is a ray in $Z_{i+1}$.
Let $\mu \in \cM_\sigma(X)$ be ergodic with $\mu(X_{\res})=1$. 
It follows from standard arguments (partitioning into sets with the same transition index $k$ and using  $\sigma$-invariance) that $\mu(X_{\lim}\setminus (Z_1\cup Z_2\cup Z_3))=0$. Using that by \cite[Lemma 4.1]{Pavlov}  every 
invariant measure assigns full measure to $X_{\con} \cup X_{\lim}$
we deduce that either $\mu(Z_1)=1$ or  $\mu(Z_2)=1$ or $\mu(Z_3)=1$ holds. 
 We conclude that there are precisely two ergodic invariant entropy maximizing probability measures among all invariant probability measures which assign full measure to $X_{\res}$. These are the $(\frac12,\frac12)$-Bernoulli measures $\mu_1$ and $\mu_2$ on 
$Z_1$ and $Z_2$, respectively. Both $\mu_1$ and $\mu_2$ have measure-theoretic entropy $\log 2$. This shows that $h_{\res}(X)=\log 2$.

Next we show that $h_{\con}(X)=\log2$. It follows from Lemma \ref{lem:entropy_formula} that $h_{\top}(X_n)=\log \lambda_n$, where $\lambda_n$ is the unique solution of the equation
$$ \sum^n_{k=1}2^{2k} \lambda^{-3k}=1.
$$ 
Moreover, $\lambda_n \uparrow \lambda_*$ as $n\to \infty$, where $\lambda_*$ is the unique solution of 
 \begin{equation}\label{probability} \sum^{\infty}_{k=1}2^{2k} \lambda^{-3k}=1.
\end{equation}
 A simple calculation shows that $\lambda_*=2$. Hence, $\lim_{n \to \infty} h_{\top}(X_n)=\log 2$ which implies $h_{\con}(X)\geq \log 2$. 
Assume that $h_{\con}(X) > \log 2$. Then  \cite[Proposition 28]{BDWY} 
leads to the following contradiction:
$$\log 2 = \lim_{n \to \infty} h_{\top}(X_n)=h_{\con}(X)>\log 2.$$
We conclude $h_{\rm top}(X)=h_{\rm con}(X)=h_{\rm res}(X)=\log 2$. In particular, $\mu_1$ and $\mu_2$ are the only ergodic measures of maximal entropy which assign full measure to $X_{\rm res}$.
It remains to consider the possibility of ergodic measures of maximal entropy which assign full measure to $X_{\con}$. 
For all $g\in \cG$ we assign $p_g=e^{-|g| \log 2}$ to the $\cG$-cylinder $\llb g\rrb$. 
We define a measure $\nu_3$ on $E$ by $\nu_3 (\llb g_{0} \cdots g_{k}\rrb)=p_{g_{0}}\cdots p_{g_{k}}$ for all $g_0,\cdots,g_k\in \cG$. Since $h_{\rm con}(X)=h_{\rm res}(X)$ we can not directly apply Theorem B. 
Instead we will show directly that the lift $\mu_3$ of $\nu_3$, that is, $\mu_3=\bl \nu_3$ is an ergodic measure of maximal entropy.

We claim that $\nu_3$ is an invariant probability measure on $E$ 
which can be lifted to $X_{\con}$. Indeed, since $|W(n)|=2^{2n}$ we obtain
$$\nu_3(E)=\sum_{g \in G} \nu_3(\llb g \rrb)=\sum_{g \in G}e^{-|g| \log 2}=\sum^{\infty}_{n=1}2^{2n}e^{-3n \log 2}=\sum_{n=1}^\infty 2^{-n}=1.$$ 
Using \eqref{eqYpiE} we observe that $\nu_3$ is measure-theoretic isomorphic to a Bernoulli measure on the two-sided countable alphabet full-shift $Y$. This shows that $\nu_3$ is shift invariant. Next we show that $\nu_3$ can be lifted to the measure $\mu_3$ 
on $X$ by showing that the normalizing constant $c$ is finite. To see this we observe that the number of elements in $\cG$ with length $3n$ is $2^{2n}$ and conclude that
$$c=\sum_{g\in \cG}|g| p_g=\sum^{\infty}_{n=1}3n\cdot 2^{2n}\cdot e^{-3n \log 2}=3\sum^{\infty}_{n=1}n 2^{-n}=6.$$ 
Hence $\mu_3$ is well-defined and $\mu_3(E)=c^{-1}=1/6$. 
 Next, we calculate the measure-theoretic entropy $h_{\sigma_E}(\nu_3)$. By applying the fact that $\nu_3$ is (up to conjugacy) a Bernoulli measure we obtain
\begin{align*}
h_{\sigma_E}(\nu_3)&=  -\sum_{g\in \cG} \nu_3(\llb g\rrb) \log(\nu_3 (\llb g\rrb )\\
&= -\sum_{g\in \cG} e^{-|g|\log 2} \log e^{-|g|\log 2}\\
&= \sum_{n=1}^{\infty} 2^{2n} 2^{-3n}3n \log 2\\
&= 3 \log 2 \sum_{n=1}^\infty n 2^{-n}= 6 \log 2.
\end{align*}
Since $\mu_3(E)=\frac16$ Abramov's formula yields 
$h_\sigma(\mu_3)=\log 2$. We conclude that $\mu_3$ is an ergodic measure of maximal entropy on $X$ and is $\cG$-Bernoulli. Having established the existence of a  measure of maximal entropy that assigns full measure to $X_{\rm con}$ its uniqueness follows from Corollary \ref{cor:h_seq=h_res}. This completes the proof.
\end{proof}

Next we consider beta shifts and analyze the corresponding  measures of maximal entropy. Our goal is to apply Corollary \ref{cor:h_seq=h_res} to derive an explicit formula for the  measure of maximal entropy. We briefly review some basic facts about beta shifts, see, e.g., \cite{BDWY,ClTh} and the references therein for more details.
Let $\beta>1$ be a real number. The {\em beta shift} $X_{\beta}$ is the natural coding space associated with the $\beta$-transformation $f_{\beta}:[0,1)\rightarrow [0,1)$ defined  by
$f_{\beta}(x)=\beta x\,\, (\text{mod } 1)$. 
If $\beta\in\bN$, then $X_\beta$ is the full shift with $\beta$ symbols.  Therefore, we may assume that $\beta$ is not an integer. 

One method to define $X_\beta$ is as  the set of paths on the countable directed labeled graph $\Gamma_\beta$. We briefly review this method.
Given $x\geq 0$ let $\lfloor x\rfloor$ denote the integer part of $x$. The beta shift $X_\beta$ is a one-sided subshift over the  alphabet $\cA=\{0,\dots,\lfloor \beta \rfloor\}$. We denote the lexicographic order on one-sided shift spaces  by $\preceq$. 
There exists a unique sequence $b=b(\beta)=(b_1b_2b_3\cdots)$ which is the lexicographic supremum over all solutions of
\begin{equation}\label{eqbeta}
\sum_{n=1}^\infty b_n\beta^{-n}<1.
\end{equation}
We note that $X_{\beta}$ is sofic if and only if $b$ is eventually periodic, see e.g. \cite{RJ}.
Thus, we may assume that $b$ is not eventually periodic. In this case $b$ could be defined by considering \eqref{eqbeta} as equality.
  The Beta shift is characterized by the following condition:
\begin{equation*}
x\in X_\beta \quad \Longleftrightarrow \quad \sigma^n(x) \preceq b(\beta)\,\, {\rm for}\,\,{\rm all}\,\, n\in \bN_0
\end{equation*}
Every beta shift can be presented as a countable directed labeled graph $\Gamma_\beta$ which is completely determined by $b=b(\beta)$.
We denote the vertices of  $\Gamma_\beta$ by $v_n$  for $n\in \bN$. For each $n\in \bN$, we add an edge from $v_n$ to $v_{n+1}$ and label it with $b_n$. 
Moreover, for all $n\in \bN$ and $i=0,\dots, b_n-1$, we add an edge  from $v_n$ to $v_1$ and label it with $i$. The beta shift $X_\beta$ coincides with the  set of  sequences of labels  associated with the set of infinite paths on $\Gamma_\beta$ that start at $v_1$.
An example of such a graph $\Gamma_\beta$ is depicted in Figure \ref{fig:beta}. 

\begin{figure}[hbt]
\begin{center}
\begin{tikzpicture}
\clip (-2.5,-1) rectangle (11.5,3);
\node [draw,circle,inner sep=0pt,minimum size=.6cm] (A) at (0,0) {$v_1$};
\node [draw,circle,inner sep=0pt,minimum size=.6cm] (B) at (2,0) {$v_2$};
\node [draw,circle,inner sep=0pt,minimum size=.6cm] (C) at (4,0) {$v_3$};
\node [draw,circle,inner sep=0pt,minimum size=.6cm] (D) at (6,0) {$v_4$};
\node [draw,circle,inner sep=0pt,minimum size=.6cm] (E) at (8,0) {$v_5$};
\node at (9.8,0) (F) {};
\draw (E) edge[->] node[below]{$1$} (F);
\draw (A) edge[->] node[below]{$2$} (B);
\draw (B) edge[->] node[below]{$2$} (C);
\draw (C) edge[->] node[below]{$0$} (D);
\draw (D) edge[->] node[below]{$1$} (E);
\draw (B) edge[->,in=20,out=160] node[above]{0} (A);
\draw (B) edge[->,in=75,out=105] node[above]{1} (A);
\draw (D) edge[->,in=75,out=105] node[above]{0} (A);
\draw (E) edge[->,in=80,out=100] node[above]{0} (A);
\draw (A) edge[->,loop left,looseness = 15,in=160,out=200] node[left]{0} (A);
\draw (A) edge[->,loop left,looseness = 20,in=140,out=220] node[left]{1} (A);
\end{tikzpicture}
\end{center}
\caption{A graph representation of a beta shift.\label{fig:beta}}
\end{figure}
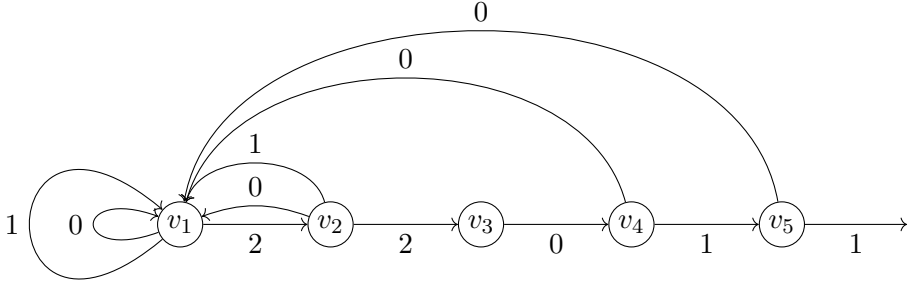

Since $X_\beta$ is a one-sided shift space, we can not directly apply our results for coded shifts.  We can, however, use the graph $\Gamma_\beta$ as a representation to associate to $X_\beta$ a  two-sided shift space $\widehat{X}_\beta$ given by all bi-infinite paths on $\Gamma_\beta$.  It follows that $\widehat{X}_\beta$
defined as the closure of the set of bi-infinite sequences of labels associated with  bi-infinite paths on $\Gamma_\beta$ is a two-sided shift space which has  the same language as $X_\beta$. It is shown in \cite{BDWY} that $\widehat{X}_\beta$ is a coded shift with generating set
\begin{equation*}
\mathcal{G_\beta}=\{g(j,i)=b_1\cdots b_{j}i: i<b_{j+1}\}\cup \{i:i<b_1\}
\end{equation*}
which uniquely represents $\widehat{X}_{\beta, {\rm con}}$. There is a natural measure-theoretic entropy preserving bijection $I:\cM_\sigma(\widehat{X}_\beta)\to \cM_\sigma(X_\beta)$ defined by $I(\mu)= \mu^+$ where $\mu^+([w])=\mu([w])$ for all $w\in \cL(X_\beta)$. Recall that $h_{\rm top}(X_\beta)=\log \beta$. It is shown in \cite{Hof} that $X_\beta$
has a unique measure of maximal entropy $\mu_{\rm max}^+$. Since $I$  is an entropy preserving bijection it follows from the variational principle that $\mu_{\rm max}=I^{-1}(\mu_{\rm max}^+)$ is the unique measure of maximal entropy of $\widehat{X}_\beta$. We have the following result.

\begin{example}\label{Example4}
Let $\beta>1$ be a non-integral real number which is not eventually periodic. Let $\widehat{X}_\beta$ denote the coded shift with generating set $\cG_\beta$ which is associated with $X_\beta$. Then $\widehat{X}_\beta$ has a unique measure of maximal 
entropy $\mu_{\rm max}$ which is $\cG_\beta$-Bernoulli with $p_g=\beta^{-|g|}$ and   
$c=\sum_{g\in \cG_\beta} |g| \beta^{-|g|}$. Moreover, $\mu^+_{\rm max}=I(\mu_{\rm max})$ is the unique measure of maximal entropy of the beta shift $X_\beta$.
\end{example}
\begin{proof}
 We already know that $\widehat{X}_\beta$ and $X_\beta$ have unique measures of maximal entropy     $\mu_{\rm max}$ and $\mu^+_{\rm max}$, respectively. We compute
\begin{align*}
 \sum_{g\in\cG_\beta}e^{-|g|h_{\rm top}(X_\beta)}& = \sum_{g\in\cG_\beta}e^{-|g|\log \beta}=\\
  \sum_{n=1}^\infty  b_n e^{-n \log \beta}&=
  \sum_{n=1}^\infty \frac{b_n}{\beta^{n}} =1, 
\end{align*}
and 
\begin{align*}
\sum_{g\in\cG}|g|e^{-|g|h_{\rm top}(X)}&= \sum_{n=1}^\infty n b_n e^{-n \log \beta}=\\
\sum_{n=1}^\infty \frac{nb_n}{\beta^{n}}&\leq\lfloor \beta \rfloor\sum_{n=1}^\infty \frac{n}{\beta^{n}}<\infty.
\end{align*}
This shows that condition (ii) in Corollary \ref{cor:h_seq=h_res} holds.
We conclude that $\mu_{\rm max}$ is the $\cG_\beta$-Bernoulli measure with $p_g=\beta^{-|g|}$ and   
$c=\sum_{g\in \cG_\beta} |g| \beta^{-|g|}$.
\end{proof}
 We note that the unique measure of maximal entropy of the beta shift $X_\beta$ corresponds to the unique absolutely continuous invariant probability measure $\nu_\beta$ of the $\beta$-transformation $f_\beta$ on the interval. A description for $\nu_\beta$ is given by Parry \cite{Parry}.

We finish this section with a comparison of Theorem \ref{thm:unique_mme} and results of Climenhaga \cite{Cl} and Pavlov \cite{Pavlov}.  
We note that definition \eqref{def:X_lim} of $X_{\rm lim}$ implies that $X_{\rm lim}$ is a shift space.
In \cite{Cl, Pavlov} the  uniqueness of the measure of maximal entropy is established under the assumption $h_{\rm top}(X_{\rm lim})<h_{\rm top}(X)$.  Similar uniqueness  results hold for equilibrium states of H\"older continuous potentials, see \cite{Cl}. We have established in Theorem \ref{thm:local_structure} examples of coded shifts $X(\cG)$ for which Theorem B guarantees the uniqueness of the measure of maximal entropy but the results of Climenhaga/Pavlov cannot be applied since $X_{\rm con}\subset X_{\lim}$ holds. 
The following observation shows that Theorem B is actually an extension of the corresponding result by Climenhaga and Pavlov.

\begin{corollary}\label{lem:ClPav}
Let $X=X(\cG)$ be a coded shift such that $\cG$ uniquely represents $X_{\rm con}$ and $h_{\rm top}(X_{\rm lim})<h_{\rm top}(X)$. Then
\begin{equation}\label{eqgen}
h_{\rm res}(X)<h_{\rm con}(X)=h_{\rm top}(X).
\end{equation}
\end{corollary}
\begin{proof}
It suffices to prove $h_{\rm res}(X)\leq h_{\rm top}(X_{\rm lim})$ for \eqref{eqgen} to hold. Let $\mu\in \cM_\sigma(X)$ with $\mu(X_{\rm res})=1$.  If follows from \cite[Lemma 4.1]{Pavlov} that $\mu(X_{\rm lim})=1.$ Therefore, $h_\sigma(\mu)\leq h_{\rm top}(X_{\rm lim})$ follows from  the variational principle for the topological entropy.
\end{proof}

\section{Uniqueness of Equilibrium States}\label{sec:unieqsta}

In this section we establish  the uniqueness of equilibrium states for hyperbolic H\"older continuous potentials. Let $(X,\cG)$ be a coded shift such that $\cG$ uniquely represents $X_{\rm con}$. Recall that
\[
h(\cG)=\limsup_{n\to\infty} \frac{1}{n}\log   c(n),
\]
where $c(n)={\rm card}\{g\in \cG: |g|=n\}.$ We consider the case $h(\cG)=0$.

 As before, $\cM_\sigma(X)$ denotes the set of all shift-invariant Borel probability measures on $X$. The topological pressure of a continuous potential $\phi:X\to\bR$ satisfies the well-known variational principle, i.e.
\begin{equation}\label{varpri}
    P_{\rm top}(\phi)=\sup\left\{h_\sigma(\mu)+\int \phi\,d\mu: \mu\in\cM_\sigma(X)\right\},
\end{equation}
see, e.g., \cite{Wal:81} for details.  We say that $\mu\in \cM_\sigma(X)$ is an equilibrium state of $\phi$
if $h_\sigma(\mu)+\int \phi\, d\mu=P_{\rm top}(\phi)$. By expansivity of the shift map,
$\phi$ has at least one equilibrium state. Moreover, by compactness and convexity of the space of equilibrium states there exists at least one ergodic equilibrium state.
We say $\phi$ is a \emph{hyperbolic potential} if every equilibrium state of $\phi$ has non-zero entropy.

We define the \emph{concatenation pressure} of a potential $\phi$ by
\begin{equation}\label{def:SecPressureseq}
  P_{\rm con}(\phi)=\sup\left\{h_\sigma(\mu)+\int \phi\,d\mu: \mu\in\cM_\sigma(X)\text{ and } \mu(X_{\rm con})=1\right\},
\end{equation}
and the \emph{residual pressure} of $\phi$ by
\begin{equation}\label{def:SecPressureres}
  P_{\rm res}(\phi)=\sup\left\{h_\sigma(\mu)+\int \phi\,d\mu: \mu\in\cM_\sigma(X)\text{ and } \mu(X_{\rm res})=1\right\}.
\end{equation}
We continue to use the notation from Section \ref{sec:mme}. In particular, we denote by $\sigma_E:E\to E$ the induced map on $E$ and by $\bi:\cM_\con\to\cM_{\rm ind}$ the bijection with inverse $\bl:\cM_{\rm ind}\to \cM_\con$ (see Lemma \ref{lem:lift}). Moreover, we fix an enumeration $\{g_i\}_{ i\in\bN}$ of $\cG$ to define the map $\pi:Y\to E$ (see Equation \eqref{eqYpiE}) which conjugates the countable alphabet two-sided full shift $(Y,\sigma)$ with $\sigma_E$.

Our main objective in this section is to prove the following result.
\begin{theorem}\label{thm:uniqueeqsta}
Let $X=X(\cG)$ be a coded shift such that $\cG$ uniquely represents $X_{\con}$ and $h(\cG)=0$. Let $\phi$ be a hyperbolic H\"older continuous potential on $X$
with
$P_{\con}(\phi)>P_{\res}(\phi)$. Then $\phi$ has a unique equilibrium state $\mu_\phi$. Moreover, $\mu_\phi$ is the lift of the invariant Gibbs measure for the associated potential on the two-sided  countable full shift.
\end{theorem}

The proof of Theorem \ref{thm:uniqueeqsta} relies on similar techniques as the proof of Theorem \ref{thm:unique_mme}. Therefore, we primarily discuss the differences in the proof.

From now on we assume that $\phi:X\to \bR$ is a H\"older continuous potential. By replacing $\phi$ with $\phi-P_{\rm top}(\phi)$ and noting that $\phi$ and $\phi-P_{\rm top}(\phi)$ have the same equilibrium states we may assume without loss of generality $P_{\rm top}(\phi)=0$. 
We start with some preliminary results.
\begin{lemma}\label{critdec2} Let $\phi$ be a hyperbolic potential with $P_{\rm top}(\phi)=0$. Then there exists  $\tau>0$ and $N \in \bN$ such that

\[S_n \phi(x) < - n \tau\]
for all $x\in X$ and all $n \geq N$.
\end{lemma}

\begin{proof}
Suppose the statement does not hold. Then there exists a sequence $\{x(m)\} \subset X$ and a strictly increasing sequence
$\{l_m\} \subset \bN$ such that
$$\frac{1}{l_m}S_{l_m} \phi (x(m)) \geq - \frac{1}{m}.$$
for all $m \in \bN$.
Let $\mu$ be a weak$^\ast$ accumulation point of the sequence of measures
$\{\mu_m\}$ given by
\[ \mu_m=\frac{1}{l_m}\sum^{l_m-1}_{k=0}\delta_{\sigma^k(x(m))}\]
Then $\mu$ is $\sigma$-invariant and $\int \phi\ d\mu \geq 0$.
It follows from the variational principle that $\mu$ is an equilibrium state with zero entropy. This contradicts
the assumption that $\phi$ is a hyperbolic potential.
\end{proof}

Given a  continuous potential $\phi:X\to \bR$, analogous as in the proof of Theorem \ref{thm:unique_mme} we define the  potential $\phi_E: E \rightarrow \bR$ by 
\begin{equation}\label{eqphiE}
\phi_E(x)= \sum^{r_E(x)-1}_{j = 0} \phi \circ \sigma^j(x),
\end{equation}
where $x=\dots g_{-2}(x)g_{-1}(x).g_{0}(x)g_{1}(x)g_2(x)\dots$ and $r_E(x)=|g_0(x)|$. Moreover, we define the induced potential $\psi:Y\to\bR$ by $\psi=\phi_E\circ \pi$.
Recall, that the $k$-variation $\phi: X \rightarrow \bR$
is given by
\[ \text{Var}_k(\phi):=\sup \{ |\phi(x)-\phi(y)|:  (x_{-k+1},...,x_{k-1})=(y_{-k+1},...,y_{k-1})\}.
\]
Following \cite{Bo}, we define
\[ \cF_X :=\{\phi: X \rightarrow \bR : \ \text{Var}_k(\phi)< c\alpha^k \text{ for some } c>0,\ \alpha \in (0,1),
\text{ all } k \in \bN \}.  \]
We use the analogous notation for the potential spaces $\cF_Y$ and $\cF_{Y^+}$.
As a direct consequence of the definition of H\"older continuity we obtain:
\begin{lemma}\label{lemHTX}
If $\phi: X \rightarrow \bR$ is H\"older continuous then $\phi \in \cF_X$.
\end{lemma}
The next Lemma shows that  $\phi \in \cF_X$ implies $ \psi \in \cF_Y$.
The following result is standard and we leave the proof to the reader.

\begin{lemma}\label{hoelderinduced}
Let $X = X(\cG)$ such that $\cG$ uniquely represents $X_\con$.
Then $\phi \in \cF_X$ implies $\psi=\phi_E\circ \pi \in \cF_Y$.
\end{lemma}

We say that $\psi,\widetilde{\psi}\in \cF_Y$ are cohomologous within the class $\cF_Y$ if there exists $u\in \cF_Y$ such that $\psi(x)=\widetilde{\psi}(x)-u(x)+u(\sigma(x))$.

We note that the results for the uniqueness of equilibrium states for countable alphabets by Mauldin and Urbanski \cite{MU}, and Sarig \cite{Sarig} are formulated for one-sided shift spaces. The following result allows the transition 
from the two-sided to the one-sided  shift with a countable alphabet.

\begin{lemma}\label{bowen} Let 
$\psi \in \cF_Y$. Then there exists  $\widetilde{\psi} \in \mathcal{F}_Y$ cohomologous to $\psi$ within the class $\cF_Y$
such that $\widetilde{\psi}(x)=\widetilde{\psi}(y)$ whenever $x_k=y_k$ for all $k\in \bN_0$.
\end{lemma}
\begin{proof} The result was first stated by Sinai \cite{Si} for subshifts of finite type over finite alphabets. A complete proof of Sinai's result is given by Bowen \cite[Lemma 1.6]{Bo}. Bowen's proof generalizes to the countable alphabet case without changes.
We only include here those details of the proof that are used later on, see 
\cite[Lemma 1.6]{Bo} for the complete proof:

Given $x\in Y$ define $r(x)$ by $r(x)_k=x_k$ for $k\geq 0$ and $r(x)_k=x_0$ for $k\leq 0$. The function $u$ is defined by 
\begin{equation}\label{equ(x)}
    u(x)=\sum_{j=0}^\infty \left(\psi(\sigma^j(x))-\psi(\sigma^j(r(x)))\right).
\end{equation}
It follows that $u\in \cF_Y$.
Moreover, $\widetilde{\psi}=\psi-u+u\circ \sigma$ has the desired properties.
This follows from $\widetilde{\psi}=\psi(r(x))+v(x)$, where 
\begin{equation}\label{eqhap}
v(x)=\sum_{j=0}^\infty \left( \psi(\sigma^{j+1}(r(x))) -\psi(\sigma^j(r(\sigma(x))))\right).
\end{equation}
\end{proof}
We note that Y. Daon presents a proof of Lemma \ref{bowen} for potentials in the Walters class rather than H\"older potentials \cite[Theorem 3.1]{Da}.
We further note that functions in $\cF_Y$ are not necessarily bounded since $Y$ has an infinite alphabet. 
For example, the unbounded function $\phi(y)=y_0$ belongs to $\cF_Y$.
The next result shows that the function $u$ defined in \eqref{equ(x)} is bounded.
\begin{corollary}\label{corubounded}
Let $\psi,\widetilde{\psi}$ and $u$ be as in the proof of Lemma \ref{bowen}. Then 
$u$ is bounded.
\end{corollary}
\begin{proof}For $x\in X$ and $j\geq 0$ we note that $\sigma^j(r(x))\in \sigma^j(x)[-j,j]$. Therefore, the result follows from Equation \eqref{equ(x)} and the fact that $\psi\in \cF_Y$.
\end{proof}

Next we obtain a version of Lemma \ref{critdec2} for the countable alphabet case.

\begin{corollary}\label{critdec} Let $\phi:X\to \bR$ be a hyperbolic potential with $P_{\rm top}(\phi)=0$ and let $\psi=\phi_E\circ \pi: Y\to \bR$. Let $\widetilde{\psi}:Y\to \bR$ be as in Lemma \ref{bowen}.
Then there exists  $\tau>0$ and $N \in \bN$ such that for all $x\in Y$ with $n=n(x)=|g_{x_0}|\geq N$ we have $\widetilde{\psi}(x) < - n \tau$. 
\end{corollary}
\begin{proof}
The result follows from $\psi=\phi_E\circ \pi$, Lemma \ref{critdec2} and Corollary \ref{corubounded}. 
\end{proof}
We note that the constants $\tau$ and $N$ in Corollary \ref{critdec} do in general differ from the corresponding constants in Lemma \ref{critdec2}.

Given $m\in \bN$ we define $Y_m=\{1,\dots,m\}^\bZ$ and consider $Y_m$ as a subshift of $Y$.
We need another auxiliary result.
\begin{lemma}\label{lem:MaUr}
Let $\psi\in \cF_Y$. Then $P_{\rm top}(Y,\psi)=\sup_{m\in \bN} P_{\rm top}(Y_m,\psi\vert_{Y_m}).$
\end{lemma}
\begin{proof}
    The result is proven in \cite[Lemma 2.1.2]{MU} for the one-sided countable alphabet case for acceptable potentials which is a weaker condition than belonging to $\cF_Y$. The proof in the two-sided countable alphabet case is entirely analogous. 
\end{proof}

Next we describe how to transition from the two-sided countable alphabet full shift $(Y,\sigma)$ to the one-sided countable alphabet full shift $(Y^+,\sigma)$ where $Y^+=\bN^{\bN_0}$. There is a natural bijection $b:\cM_\sigma(Y)\to \cM_\sigma(Y^+)$ defined by $\nu^+=b(\nu)$ where $\nu^+([w])=\nu([w])$ for all $w\in \bN^n, n\in \bN$. Clearly, $\nu^+$ is well-defined since  invariant measures are uniquely determined by their values on cylinders. Further, Equation \eqref{eqn:def:meas_ent} shows that
\begin{equation}
h_{\sigma}(\nu)=h_{\sigma}(\nu^+).
\end{equation}
Let $\psi$ be a continuous potential on $Y$ with $\psi(x)=\psi(y)$ whenever $x_k=y_k$ for all $k\geq 0$. We  define a continuous potential $\psi^+$ on $Y^+$ by
$\psi^+(y^+)=\psi(y)$ where $y\in Y$ with $y_k=y^+_k$ for all $k\geq 0$. It follows directly from the definitions that  
\begin{equation}
    \int \psi\, d\mu = \int \psi^+\, d\mu^+.
\end{equation}
Hence,
\begin{equation}\label{eqY2Y+}
h_{\sigma}(\nu)+\int \psi\, d\nu=h_{\sigma}(\nu^+)+\int \psi^+\, d\nu^+.
\end{equation}

Let $\psi$ be a continuous potential on $Y$ with $P_{\rm top}(Y,\psi)<\infty$. Following \cite{MU} we say $\nu\in \cM_\sigma(Y)$ is an equilibrium state of $\psi$ if $\int -\psi \,d \nu <\infty$ and $P_{\rm top}(Y,\psi)=h_\sigma(\nu)+\int \psi d\, \nu.$ Analogously, we define equilibrium states for continuous potentials on $Y^+$. We have the following:

\begin{lemma}\label{lemeqYYp}
Let $\psi:Y\to \bR$ be a continuous potential with $P_{\rm top}(Y,\psi)<\infty$ and $\psi(x)=\psi(y)$ whenever $x_k=y_k$ for all $k\geq 0$. Then $P_{\rm top}(Y,\psi)=P_{\rm top}(Y^+, \psi^+)$.
Moreover, $\nu\in \cM_\sigma(Y)$ is an equilibrium state of $\psi$ if and only if $\nu^+=b(\nu)$ is an equilibrium state of $\psi^+$.
\end{lemma}
\begin{proof}
 The statements follow from the definition of the topological pressure and Equation \eqref{eqY2Y+}.
\end{proof}
We note that Lemma \ref{lemeqYYp} does not imply the existence of an equilibrium state of $\psi$ and $\psi^+$, respectively. We are now ready to present the proof of Theorem \ref{thm:uniqueeqsta}.

\begin{proof}[Proof of Theorem \ref{thm:uniqueeqsta}]

Let $\phi:X\to\bR$ be a hyperbolic potential. Without loss of generality (by replacing $\phi$ with $\phi-P_{\rm top}(\phi)$) we may assume $P_{\rm top}(\phi)=0$. Let $\psi=\phi_E\circ \pi$ be the induced potential of $\phi$. It follows from Lemmas \ref{lemHTX} and \ref{hoelderinduced} that $\psi\in \cF_Y$.
Let $\widetilde{\psi}\in \cF_Y$ be the potential associated with $\psi$, as in Lemma \ref{bowen}, i.e., $\widetilde \psi(x)$ does not depend on $x_k$ for $k<0$. Moreover, let $\widetilde{\psi}^+$ the potential on $Y^+$ associated with $\widetilde{\psi}$. Hence, $\widetilde{\psi}^+\in \cF_{Y^+}$.
Let $I:\cM_{\rm con}\to \cM_\sigma(Y^+)$ be defined by $I(\mu)=((\pi^{-1})_*\mu_E)^+$. 
It follows from Lemma \ref{lem:lift} that $I$ is injective. Here we also use the facts that $(\pi^{-1})_*$ and $\nu\mapsto \nu^+$ are injective maps on $\cM_{\rm ind}$ and $\cM_\sigma(Y)$, respectively. 
Let $\mu\in \cM_{\rm con}$ and $\nu=(\pi^{-1})_* \mu_E$. 
It follows from Abramov's formula, Lemma  \ref{bowen} and the definition of $\nu^+$ that
\begin{equation}\label{eq:freeenergy}
\begin{split}
h_\sigma(\nu)+\int \psi\,d\nu=    h_\sigma(\nu)+\int \widetilde{\psi}\,d\nu=
h_\sigma(\nu^+)+\int \widetilde{\psi}^+\,d\nu^+\\
=\frac{1}{\mu(E)}\left(h_\sigma(\mu)+\int \phi\,d\mu\right).
\end{split}
\end{equation}
Let $\mu_{\rm \phi}\in\cM_\sigma(X)$ be an ergodic equilibrium state of $\phi$. We will show that $\mu_{\rm \phi}$ is the unique equilibrium state of $\phi$. Since the space of $\phi$-equilibrium states is compact and convex, it suffices to show that $\mu_{\rm \phi}$ is the unique ergodic equilibrium state of $\phi$.
It follows from $P_{\rm con}(\phi)>P_{\rm res}(\phi)$ that $\mu_\phi(X_{\rm con})=1$, that is, $\mu_{\rm \phi}\in \cM_{\rm con}$. 

We claim that $P_{\rm top}(Y,\psi)=P_{\rm top}(Y,\widetilde{\psi})=P_{\rm top}(Y^+,\widetilde{\psi}^+)=0.$

To prove the claim we first show $P_{\rm top}(Y^+,\widetilde{\psi}^+)=0$.
Let $\nu_{\rm \phi}^+=I(\mu_{\rm \phi})$. Thus, by \eqref{eq:freeenergy}, $h_\sigma(\nu_{\rm \phi}^+)+\int \widetilde{\psi}^+\,d\nu_{\rm \phi}^+=0$. Therefore, $P_{\rm top}(Y^+,\widetilde{\psi}^+)\geq 0$ follows from the variational principle for countable alphabet one-sided shift spaces, see \cite[Theorem 2.1.8]{MU}.
The inequality $P_{\rm top}(Y^+,\widetilde{\psi}^+)\leq 0$ follows analogously as in the proof of Theorem \ref{thm:unique_mme} (also see \cite[Proposition 29]{BDWY}) by approximating $P_{\rm top}(Y^+,\widetilde{\psi}^+)$ with the pressure 
$P_{\rm top}(Y^+_m,\widetilde{\psi}^+|_{Y^+_m})$, where $Y^+_m=\{1,\dots,m\}^{\bN_0}$, and taking the limit $m\to\infty$. We conclude that  $P_{\rm top}(Y^+,\widetilde{\psi}^+)=0$ and that $\nu_{\rm \phi}^+$ is an equilibrium state of $\widetilde{\psi}^+$. The identity $P_{\rm top}(Y,\widetilde{\psi})=P_{\rm top}(Y^+,\widetilde{\psi}^+)$ holds by Lemma \ref{lemeqYYp}. Finally, the identity $P_{\rm top}(Y,\psi)=P_{\rm top}(Y,\widetilde{\psi})$ follows by combining Lemma \ref{lem:MaUr} with the fact that $P_{\rm top}(Y_m,\psi)=P_{\rm top}(Y_m,\widetilde{\psi})$ since the potentials $\psi$ and $\widetilde{\psi}$ are cohomologous. This proves the claim.

Since $v_{\rm \phi}^+=I(\mu_{\rm \phi})$ is an ergodic equilibrium state of the potential $\widetilde{\psi}^+$ and since the map $I$ is injective it suffices to show that
that $\widetilde{\psi}^+$ has only one equilibrium state.
Let $\nu^+_{\rm Gibbs}$ be the unique invariant Gibbs measure of $\widetilde{\psi}^+$ of $Y$. The existence of this measure follows from applying \cite[Theorem 1]{Sarig} to $\widetilde{\psi}^+\in \cF_{Y^+}$. 
Let $N$ and $\tau$ be as in Corollary \ref{critdec} and let $c(n)=|\{g_i: |g_i|=n\}|.$ Since $h(\cG)=0$, we may assume (by making $N$ larger if necessary) that $c(n) \leq \exp(\frac{n\tau}{2})$ for all $n \geq N$. 
For $i\in\bN$ we write
\[
a_i=\inf \left(- \widetilde \psi^+|_{[i]}\right)\exp\left(\inf \left( \widetilde \psi^+|_{[i]}\right)\right).\]  
Let $A=\sum_{i=1}^{\infty}a_i.$
Recall that $a_i> 0$ whenever $|g_i|\geq N$. Therefore, we may rewrite $A$ as

 \begin{equation}
 A=\sum^{\infty}_{n=1} \sum_{|g_i|=n}a_i=\sum^{N-1}_{n=1} \sum_{|g_i|=n}a_i+\sum^{\infty}_{n=N} \sum_{|g_i|=n}a_i.
 \end{equation}
Clearly, $\sum^{N-1}_{n=1} \sum_{|g_i|=n}a_i<\infty$.
Moreover, it follows from $\widetilde{\psi}=\psi-u+u\circ \sigma$, the definition of $\widetilde \psi^+$ and Equation \eqref{eqphiE} that
\begin{align*}
\sum^{\infty}_{n=N} \sum_{|g_i|=n}a_i&\leq \sum_{n=N}^\infty   (n \max |\phi| +2\max |u|) c(n) \exp(-n\tau)\\
&\leq  \sum_{n=N}^\infty (n \max |\phi|+2\max |u|) \exp\left(-\frac{n\tau}{2}\right).
\end{align*}
We conclude that $A < \infty$. Thus, by \cite[Lemma 2.2.8]{MU},  $\int -\widetilde{\psi}^+\, \nu^+_{\rm Gibbs}<\infty$, and \cite[Theorem 2.2.9]{MU} implies that $\nu^+_{\rm Gibbs}$ is the unique equilibrium state of $\widetilde{\psi}^+$. In particular, $\nu^+_{\rm Gibbs}=\widetilde{\nu}^+_{ \phi}$. Since $\widetilde{\nu}^+_{\phi}$
is the unique invariant Gibbs measure of $\widetilde \psi^+$ it follows that $\nu_{\phi}=(\pi^{-1})_*\mu_{\phi,{\rm E}}$ is the invariant Gibbs measure of $\psi$. This shows that the unique equilibrium state $\mu_{\rm \phi}$ of $\phi$ is the lift of a Gibbs measure on $Y$.

\end{proof}

 Climenhaga \cite{Cl} showed the uniqueness of equilibrium states for H\"older continuous potentials $\phi$ under the assumption $P_{\rm top}(X_{\rm con},\phi)>P_{\rm top}(X_{\rm lim},\phi)$. 
Below we establish the existence of abundantly many coded shift spaces $X$ satisfying ${X}_{\rm con} \subset  X_{\lim} $. In particular,  Climenhaga's result cannot be applied to derive the uniqueness of equilibrium states of H\"older continuous potentials on $X$. The proof of next result follows along the lines of the proof of Theorem \ref{thm:local_structure}.

\begin{corollary}\label{cor:local_structure} 
Let $X=X(\cG)$ be a coded shift so that
$\cG$ uniquely represents $X_{\con}$ with $h(\cG)=0$. Then there exists a generating set $\widetilde{\cG}=\cG\cup \{f_i: i\in\bN\}$ such that the coded shift $\widetilde X = X(\widetilde \cG)$ has the following properties:
\begin{enumerate}
\item[(i)]   $\widetilde \cG$ uniquely represents $\widetilde X_{\con}$, and $h(\widetilde \cG)=0$;
\item[(ii)] $\widetilde{X}_{\rm con} \subset \widetilde X_{\lim} $.
\end{enumerate}
\end{corollary}

 Let $\widetilde X=X(\widetilde \cG)$ be as in Corollary \ref{cor:local_structure}. Then $h(\widetilde G)=0$ and $\widetilde X_{\rm con}\subset \widetilde X_{\rm lim}$. By Theorem \ref{thm:uniqueeqsta}, every hyperbolic H\"older continuous potential $\phi$ with $P_{\rm top}(X_{\rm con},\phi)>P_{\rm top}(X_{\rm res}, \phi)$ has a unique equilibrium state $\mu_\phi$. This result cannot be deduced from \cite{Cl}.

Finally, we discuss the Dyck shift as an example where neither the results of Climenhaga \cite{Cl} nor Theorem \ref{thm:uniqueeqsta} can be applied to establish the uniqueness of equilibrium states for any H\"older continuous potential. However, our techniques can still be used to characterize the ergodic measures of maximal entropy. We recall the simple description of this shift,
given in terms of parentheses and brackets. The alphabet of cardinality four is split into two
pairs of matching left and right symbols denoted by (, ), [, ]. The Dyck shift is the smallest shift space which contains the set of bi-infinite sequences where every opening parenthesis ( must be closed by ) and every opening bracket [ must be closed by ]. Krieger \cite{Krieger} showed that the Dyck shift has topological entropy $\log 3$ and admits exactly two ergodic measures of maximal entropy, both are fully supported and Bernoulli. It is known that the Dyck shift is coded. We describe below its canonical generating set $\cG$ and compute the exact value of $h(\cG)$. By using an alternative coding we are able to provide a simple expression for its measures of maximal entropy.

\begin{example}\label{ex:Dyck} The Dyck shift is a coded shift with the generating set $\cG=\bigcup\limits_{n=1}^\infty W_n$ where the sets of words $W_n$ are recursively defined by
\begin{align*}
  W_1 & =\{(\,),[\,]\} \\
  W_n & =\{(w),[w]:w\in(\cup_{i=1}^{n-1}W_i)^*\text{ and } |w|=2(n-1)\}.
\end{align*}
Moreover, $\cG$ uniquely represents $X_\con$ and $h(\cG)=\frac32\log 2$. Let $\cG_1=\cG\cup\{(,[\,\}$ and $\cG_2=\cG\cup\{\,),]\}$. Then the two ergodic measures of maximal entropy of $X$ are  $\cG_1$- respectively $\cG_2$-Bernoulli with $p_g=3^{-|g|}$ and $c=2$.
\end{example}
\begin{proof}
  It was noted in \cite[Example 5.5]{Pavlov} that the canonical generator $\cG$ uniquely represents $X_\con$. Clearly, in this case $X_{\rm lim}(\cG)=X(\cG)$. The residual set consists of elements from $\left(\cG\cup\{(,[\,\}\right)^*$ and $\left(\cG\cup\{\,),]\}\right)^*$ where at least one bracket/parenthesis is present, as well as ``symmetric" concatenations $x(-\infty,k)y[k,\infty)$ where $x\in\left(\cG\cup\{(,[\}\right)^*$, $y\in\left(\cG\cup\{),]\}\right)^*$ and the order of open parentheses and brackets between the generating words in $x(-\infty,k)$ moving to the left is the same as the order of corresponding closed ones in $y[k,\infty)$ moving to the right.

  Next, we compute the entropy of the generating set $\cG$. The formula $$|W_n|=\frac{2^{n-1}(2n)!}{(2n-1)(n!)^2}$$ is derived using standard methods, which we briefly outline below since we require some of the ingredients later on. Denote by $d_n$ the cardinality of the set $W_n$. It follows from the definition of $W_n$ that
  $$
  d_{n+1}=2\sum_{k=1}^{n}\sum_{i_1+...+i_k=n}d_{i_1}d_{i_2}...d_{i_k}.
  $$
  If $d_{i_k}=j$ for some $j\in\{1,...,n\}$ we can partition the remainder $(n-j)$ into $k$ parts where $k$ is at most $n-j$. We obtain
  $$
  d_{n+1}=2\sum_{j=1}^{n}d_{j}\sum_{k=1}^{n-j}\sum_{i_1+...+i_k=n-j}d_{i_1}d_{i_2}...d_{i_k}=\sum_{j=1}^{n}d_{j}d_{n-j+1}.
  $$
  To solve the recursive relation above we use the generating function $F(x)=\sum_{n=1}^{\infty}d_nx^n$. We recall that $d_1=2$ and compute
    $$F(x) = 2x+\sum_{n=1}^{\infty}d_{n+1}x^{n+1} = 2x + \sum_{j=1}^{\infty}\sum_{n=j}^{\infty} d_jx^j d_{n-j+1}x^{n-j+1}= 2x+F(x)^2. $$
  We obtain $F(x)^2-F(x)+2x=0$, so that $F(x)=\frac{1-\sqrt{1-8x}}{2}$ (we picked the negative sign since $F(0)=0$). We use the binomial formula
  $$\sqrt{1-8x}=\sum_{n=0}^{\infty}\frac{(-1)^{n+1}}{4^n(2n-1)}\binom{2n}{n}(-8x)^n,$$
  and get
  $$ F(x) = \frac12\left(1-\sqrt{1-8x}\right)=\sum_{n=1}^{\infty}\frac{2^{n-1}}{2n-1}\binom{2n}{n}x^n.$$
  Therefore, $d_n=\frac{2^{n-1}}{2n-1}\binom{2n}{n}=\frac{2^{n-1}(2n)!}{(2n-1)(n!)^2}$, as required.

  Since the length of words in $W_n$ is $2n$, to compute the entropy of the generating set we need to evaluate $\lim\limits_{n\to\infty}\frac{\log |W_n|}{2n}$. We apply Stirling's formula $n!\approx \frac{n^n}{e^n}\sqrt{2\pi n}$. Since $\log d_n \approx 3n\log 2-\log 2-\log(2n-1)+\log\sqrt{4\pi n}-2\log\sqrt{2\pi n}$,
  we see that $$\lim_{n\to\infty}\frac{\log d_n}{2n}=\frac32\log 2.$$

  Now consider the generating set $\cG_1=\cG\cup\{(,[\,\}$ and the corresponding coded shift $X(\cG_1)$. It follows from the structure of the residual set with respect to the canonical generator $\cG$ that $X=X(\cG)=X(\cG_1)$. We need to verify that $\cG_1$ uniquely represents $X_\con(\cG_1)$. First, we show that no two generators in $\cG$ can overlap. To the contrary, assume that generators $u=u_1...u_n$ and $v=v_1...v_m$ in $\cG$ overlap, i.e. $u_i=v_1,u_{i+1}=v_2,...,u_n=v_{n-i+1}$ for some $1<i<n$. Then $u_i$ is the open parenthesis (or bracket) which must have a corresponding closed one within the symbols $u_{i+1},...,u_{n-1}$ by the property of a generating word $u$. However, $u_i=v_1$ and the corresponding closed parenthesis is $v_m$ which is not in $u$. Now suppose that there is $x\in X_\con(\cG_1)$ which can be represented as two different concatenations of elements from $\cG_1$. Then there are two increasing sequences of integers $(k_i)_{i\in\bZ}$ and $(m_j)_{j\in\bZ}$ such that $x[k_i,k_{i+1}),x[m_j,m_{j+1})\in \cG_1$ and for some fixed $i,j$ we have $k_i\le m_j<k_{j+1}<m_{j+1}$. Since words in $\cG$ do not overlap, we must have $x[k_i,k_{i+1})\in\{(,[\,\}$ so that $k_i=m_j$ and $k_{i+1}=k_i+1$. Let $n=\max\{l\ge j:k_{l+1}=k_l+1\}$. Note that $k_{n+1}<m_{j+1}-1$ since $x_{k_{n+1}}$ is an open parenthesis/bracket and the word $x[m_j,m_{j+1})$ must end with a closed parenthesis/bracket. The words $x[k_{n+1},k_{n+2})$ and $x[m_j,m_{j+1})$ are both in $\cG$ and overlap, which is a contradiction.

  Recall that $h_{\rm top}(X)=\log 3$ and $d_n={\rm card}\{g\in\cG_1: |g|=2n\}$. We evaluate
 $$
    \sum_{g\in\cG_1}e^{-|g|\log 3}  =2e^{-\log 3}+\sum_{n=1}^{\infty}d_n e^{-2n\log 3} = \frac{2}{3}+ F(3^{-2}) = 1.
 $$
Moreover, 
$$c=\sum_{g\in\cG_1}|g|e^{-|g|\log 3}=2e^{-\log 3}+2\sum_{n=1}^{\infty}nd_n e^{-2n\log 3}=
\frac{2}{3}+2(xF'(x))|_{x=3^{-2}}.$$
By the formula for $F$ above $F'(x)=2(1-8x)^{-1/2}$, hence $F'(3^{-2})=6$ and $c=2$. 
 
 By Corollary \ref{cor:h_seq=h_res}, the $\cG_1$-Bernoulli measure $\mu_1$ with probabilities $p_g=3^{-|g|}$ and constant $c=2$ is an ergodic measure of maximal entropy on $X$. Applying the same argument to the generating set $\cG_2$ we see that the second ergodic measure of maximal entropy $\mu_2$ is $\cG_2$-Bernoulli with the same probabilities $p_g$ and constant $c$. \end{proof}

 Note that in the previous example $h_\con(X(\cG_1))=h_\res(X(\cG_1))$ and the coded shift $X(\cG_1)$ has two ergodic measures of maximal entropy, $\mu_1$ and $\mu_2$, where $\mu_1$ assigns full measure to $X_\con(\cG_1)$ and $\mu_2$ assigns full measure to $X_\res(\cG_1)$. This gives an example for the non-uniqueness of measures of maximal entropy for a coded shift whose residual and concatenation entropies are equal.

\section*{Acknowledgements} 
The authors express their most sincere gratitude to the anonymous referee for many helpful suggestions which
greatly improved the paper, and clarified the proofs  and  results.
The authors further thank Yun Yang for many comments and discussions which greatly contributed to the paper. This work was significantly
advanced while the authors were supported by the Research in Pairs program at the Mathematisches Forschungsinstitut Oberwolfach. T. Kucherenko was partially supported by a grant from Simons Foundation (\#855117 to T. Kucherenko). M. Schmoll was partially supported by a grant from the Simons Foundation (\#318898 to Martin Schmoll). C. Wolf was partially supported by  grants from  the Simons Foundation (\#637594 to Christian Wolf) and PSC-CUNY (\#6D132-00 03 to Christian Wolf).

\bibliographystyle{amsplain}

\begin{thebibliography}{99}

\bibitem{aar} J. Aaronson, \emph{ An introduction to infinite ergodic theory}, Mathematical Surveys and Monographs
vol {\bf 50}, American Mathematical Society (1997), 284pp.

\bibitem{ADJS} D. Amahdi Dastjerdi and S. Jangjooye Shaldehi, \emph{$(S,S')$-gap shifts as a generalization of
run-length-limited codes}, British Journal of Mathematics and Computer Science, {\bf 4} (2014),
2765--2780.

\bibitem{BPR} M. B\'{e}al, D. Perrin, and A. Restivo, \emph{Unambiguously coded systems}, European Journal of Combinatorics, to appear.

\bibitem{BH}F. Blanchard and G. Hansel, \emph{Syst$\grave{e}$mes cod\'es,} Theoretical Computer Science {\bf 44} (1986), 17--49.

\bibitem{Bo} R. Bowen, \emph{Equilibrium States and the Ergodic Theory of Anosov Diffeomorphisms}, Lecture Notes in Mathematics {\bf 470}, Springer-Verlag, Berlin, 2008, viii+75 pp.

\bibitem{BDWY}
M. Burr, S. Das, C. Wolf, and Y. Yang, \emph{Computability of topological pressure on compact shift spaces beyond finite type}, Nonlinearity {\bf 35} (2022) 4250.

\bibitem{jB05} J. Buzzi, \emph{Subshifts of quasi-finite type}, Inventiones Mathematicae {\bf 159} (2005), 369--406.

\bibitem{Cl} V. Climenhaga, \emph{Specification and towers in shift spaces}, Comm. Math. Phys. {\bf 364} (2018), 441--504.

\bibitem{Cl2}V. Climenhaga, Entropy-bounds-for-equilibrium-states, 
\url{https://vaughnclimenhaga.wordpress.com/2017/01/26/entropy-bounds-for-equilibrium-states/} (2017)

\bibitem{ClTh} V. Climenhaga and D. Thompson, \emph{Intrinsic ergodicity beyond specification: $\beta$-
shifts, S-gap shifts, and their factors}, Israel J. Math. {\bf 192} (2012), no. 2, 785--817.

\bibitem{ClTh2} V. Climenhaga and D. Thompson, \emph{Equilibrium states beyond specification and the Bowen property},
Journal of the London Mathematical Society {\bf 87} (2013), 401--427.

\bibitem{Da}Y. Daon, \emph{Bernoullicity of equilibrium measures on countable Markov shifts},
Discrete and Continuous Dynamical Systems {\bf 33} (2013), 4003--4015.

\bibitem{Dillon} T. Dillon, \emph{Dynamics and entropy of $ \mathcal{S} $-graph shifts}, Discrete and Continuous Dynamical Systems, to appear.

\bibitem{FF}D. Fiebig and U. Fiebig, \emph{Covers for coded systems,} In Peter Walters, editor, Symbolic Dynamics and its Applications, volume 136 of Contemporary Mathematics, American Mathematical Society, 1992.

\bibitem{GRP} F. Garc\'{\i}a-Ramos and R. Pavlov, \emph{Extender sets and measures of maximal entropy for subshifts},
Journal of the London Mathematical Society, {\bf 100} (2019), 1013--1033.


\bibitem{gu} B. M. Gurevi\v{c}, \emph{Shift entropy and Markov measures in the space of paths of a countable graph}
(Russian). Dokl. Akad. Nauk SSSR {\bf 192} (1970), 963--965.


\bibitem{H} N. Haydn, \emph{Phase transitions in one-dimensional subshifts}, Discrete and Continuous Dynamical Systems {\bf 33} (2013) 1965--1973.

\bibitem{Hof}F. Hofbauer, \emph{$\beta$-Shifts have unique maximal measure}, Monatshefte f\"ur Mathematik {\bf 85} (1978), 189--198.

\bibitem{RJ}R. Johansen,
\emph{Flow equivalence of sofic beta-shifts}, Ergodic Theory and Dynamical  Systems {\bf 37} (2017), 786--801.

\bibitem{Krieger} W. Krieger, \emph{On the uniqueness of the equilibrium state}, Math. Systems Theory {\bf 8} (1974), 97--104.



\bibitem{LindMarcus} D. Lind and B. Marcus, \emph{An Introduction to Symbolic Dynamics and Coding}, Cambridge
University Press, New York, NY, 1989.

\bibitem{MS} B. Matson and E. Sattler, \emph{$S$-limited shifts}, Real Analysis Exchange {\bf 43} (2018), 393--416.


\bibitem{MU} D. Mauldin and M. Urbanski,   \emph{Graph Directed Markov Systems: Geometry and Dynamics of Limit Sets},  Cambridge Tracts in Mathematics, Series Number {\bf 148}, Cambridge University Press, (2003), 1st edition.

\bibitem{Parry} W. Parry,  \emph{On the $\beta$-expansions of real numbers}, Acta Math. Acad. Sci. Hungar., {\bf 11} (1960),
401--416. 


\bibitem{Pavlov} R. Pavlov, \emph{On entropy and intrinsic ergodicity of coded shifts}, Proceedings of the American Mathematical Society, {\bf 148} (2022), 4717--4731.

\bibitem{Petersen} K. Petersen, \emph{Chains, entropy, coding}, Ergodic Theory and Dynamical Systems, {\bf 6} (1986), 415--448.


\bibitem{Sarig}  O. Sarig, \emph{Existence of Gibbs measures for countable Markov shifts}, Proceedings of
the American Mathematical Society,  {\bf 131} (2003), 1751--1758.

\bibitem{Si}Y.G. Sinai, \emph{Gibbs measures in ergodic theory}, Uspehi Mat. Nauk, {\bf 27} (1972), 21--64.

\bibitem{Spandl} C. Spandl, \emph{Computing the topological entropy of shifts}, Mathematical Logic Quarterly, {\bf 53} (2007), 493--510.


\bibitem{vj} D. Vere-Jones, \emph{ Ergodic properties of nonnegative matrices.} I, Pacific J. Math {\bf 22} (1967),
361--386.

\bibitem{Sh} M. S. Shtilman, \emph{On the number of invariant measures with maximal entropy for the shift in a sequence space}, Mat. Zam. {\bf 9} (1971), 291--302 .


\bibitem{Wal:81} P.~Walters, \emph{An introduction to ergodic theory},
  Graduate Texts in Mathematics 79, Springer, 1981.
  
\bibitem{Wei1}B. Weiss, \emph{Intrinsically ergodic systems}, Bulletin American Mathematical Society  {\bf 76} (1970), 1266--1269. 

\bibitem{Wei2} B. Weiss, \emph{Subshifts of finite type and sofic systems}, Monatsh. Math. {\bf 77} (1973), 462--474.

\end{thebibliography}

\end{document}